\documentclass{amsart}
\usepackage{amsmath}
\usepackage{amssymb}
\usepackage{bbm}
\usepackage{mathdots}
\usepackage{dsfont}
\usepackage{graphicx}
\usepackage{tikz}
\usetikzlibrary{calc}
\usepackage{subcaption}
\usepackage{comment}
\usepackage{soul}

\def\CC{{\mathbb C}} 
\def\RR{{\mathbb R}}

\theoremstyle{definition}
\newtheorem{definition}{Definition}[section]
\newtheorem{example}[definition]{Example}

\theoremstyle{plain}
\newtheorem{theorem}[definition]{Theorem}
\newtheorem{lemma}[definition]{Lemma}
\newtheorem{proposition}[definition]{Proposition}
\newtheorem{corollary}[definition]{Corollary}
\newtheorem*{corollary*}{Corollary}

\newtheorem{question}{Question}

\usepackage[draft]{fixme}
\fxsetup{theme=color, mode=multiuser, noinline}
\FXRegisterAuthor{rt}{RT}{\color{blue}Rekha}
\FXRegisterAuthor{ss}{TD}{\color{teal}Stefan}
\FXRegisterAuthor{jg}{JG}{\color{purple}Joao}

\begin{document}

\title{Conformal Rigidity and Spectral Embeddings \\of Graphs}

\author[]{Jo\~ao  Gouveia}
\address{CMUC, Department of Mathematics, University of Coimbra, 3001-454 Coimbra, Portugal} 
\email{jgouveia@mat.uc.pt}
\thanks{J.G. is partially supported by Centro de Matem\'atica da Universidade de Coimbra (CMUC) - UID/MAT/00324. 
}

\author[]{Stefan Steinerberger}
\address{Department of Mathematics, University of Washington, Seattle, WA 98195, USA} 
\email{steinerb@uw.edu}

\author[]{Rekha R. Thomas}
\address{Department of Mathematics, University of Washington, Seattle, WA 98195, USA} 
\email{rrthomas@uw.edu}

\begin{abstract}
    We investigate the structure of conformally rigid graphs. Graphs are conformally rigid if introducing edge weights cannot increase (decrease) the second (last) eigenvalue of the Graph Laplacian. Edge-transitive graphs and distance-regular graphs are known to be conformally rigid. We establish new results using the connection between conformal rigidity and edge-isometric spectral embeddings of the graph. All $1$-walk regular graphs are conformally rigid, a consequence of a stronger property of their embeddings. Using symmetries of the graph, we establish two related characterizations of when a vertex-transitive graph is conformally rigid. This provides a necessary and sufficient condition for a Cayley graph on an abelian group to be conformally rigid. As an application we exhibit an infinite family of conformally rigid circulants. Our symmetry technique can be interpreted in the language of semidefinite programming which provides another criterion for conformal rigidity in terms of edge orbits. The paper also describes a number of explicit conformally rigid graphs whose conformal rigidity is not yet explained by the existing theory.
\end{abstract}

\maketitle

\section{Introduction}
\subsection{Introduction.}
The purpose of this paper is to study graphs that exhibit a certain type of spectral symmetry known as \textbf{conformal rigidity} (first introduced in \cite{steinthomas}). We start with the relevant terms and motivate the notion.
Let $G=(V,E)$ be a finite, connected, undirected, unweighted graph on $n = |V|$ vertices $\left\{v_1, \dots, v_n\right\}$. Any such graph admits a Graph Laplacian
$$ L = D - A,$$
where $D \in \mathbb{R}^{n \times n}$ is the diagonal matrix, $D_{ii} = \deg(v_i)$, and
$A \in \mathbb{R}^{n \times n}$ is the adjacency matrix encoding the edge structure
$$ A_{ij} = \begin{cases}
    1 \qquad &\mbox{if}~(i,j) \in E \\
    0 \qquad &\mbox{otherwise.}
\end{cases}$$

The matrix $L \in \mathbb{R}^{n \times n}$ is symmetric and positive semidefinite. If $x:V \rightarrow \mathbb{R}$ is a vector, then the quadratic form induced by $L$ is very instructive, it is
$$ \left\langle x, Lx \right\rangle = \sum_{(i,j) \in E} (x_i - x_j)^2.$$
This quantity is sometimes known as the {\em Dirichlet energy}. If the graph is connected, which we always assume throughout the paper, then $L$ has eigenvalue 0 with multiplicity one, the corresponding eigenvector is the constant vector $\mathbf{1} \in \mathbb{R}^n$. The second smallest eigenvalue $\lambda_2$ is known as the \textbf{algebraic connectivity} of the graph. It plays a central role in Spectral Graph Theory and is intimately connected to the overall connectivity of the graph.

\begin{center}
    \begin{figure}[h!]
    \begin{tikzpicture}
            \node at (-4,0) {\includegraphics[width=0.3\textwidth]{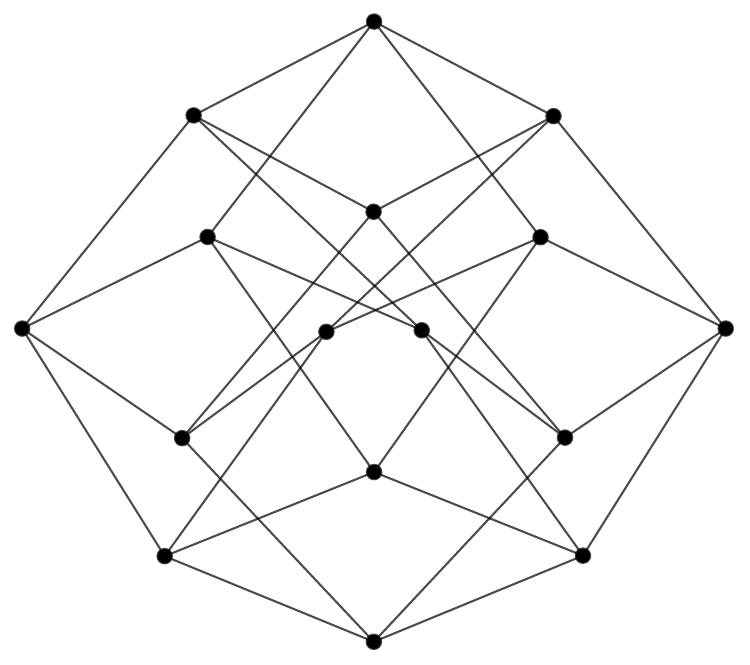}};
        \node at (0,0) {\includegraphics[width=0.25\textwidth]{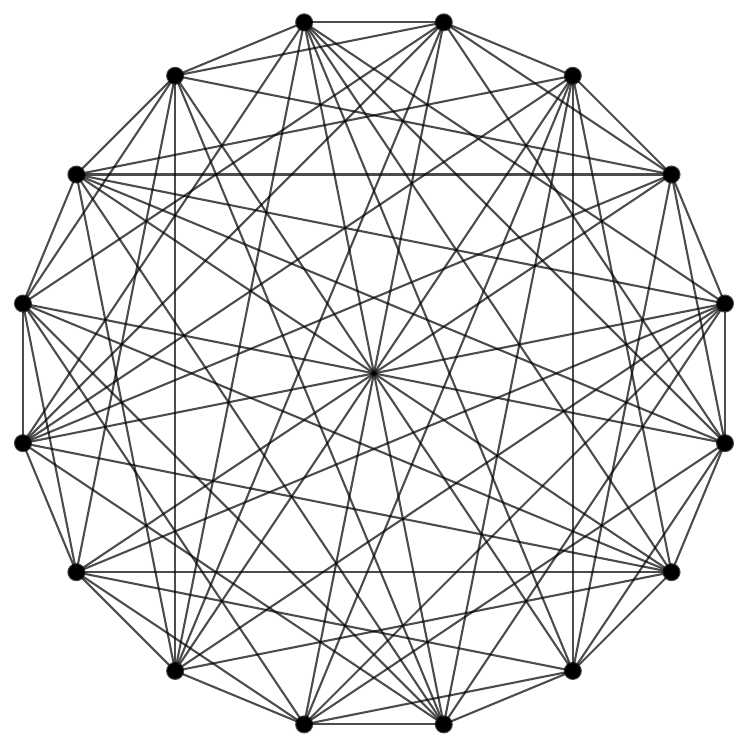}};
        \node at (4,0) {\includegraphics[width=0.25 \textwidth]{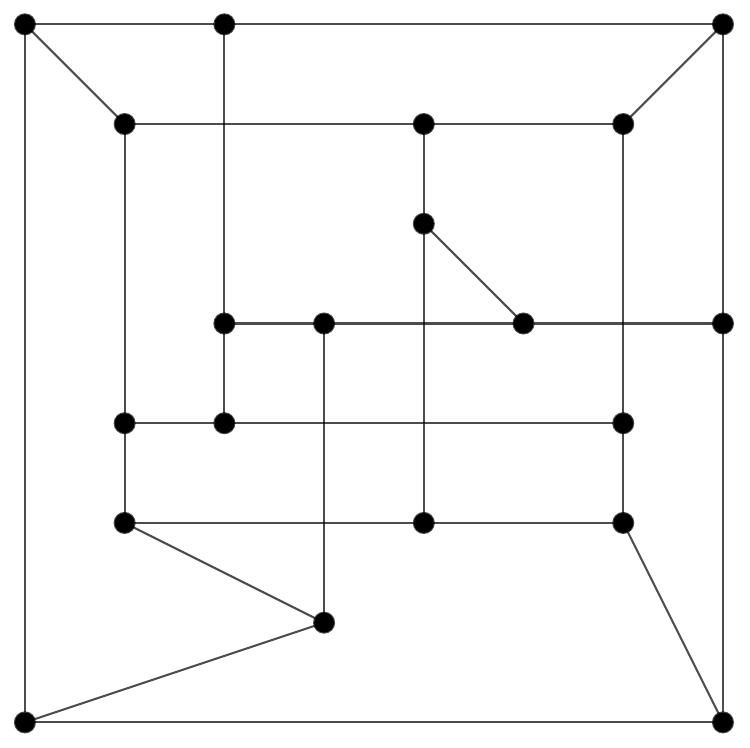}};
    \end{tikzpicture}
    \caption{Three conformally rigid graphs. Left: Hoffman graph. Middle: complement of the Shrikhande graph. Right: CNG 6B.}
    \label{fig:conformally rigid}
    \end{figure}
\end{center}

Motivated by that central role, there is a natural question of whether it is possible to increase $\lambda_2$ by changing the weights on the edges (since we start with combinatorial graphs, the initial (implicit) weight on each edge is 1); we will always assume these weights to be non-negative. If we assume that each edge $(i,j) \in E$ has a weight $w_{ij} = w_{ji} \geq 0$, then here is a natural generalization of the Laplacian: the degree is replaced by the sum of the weights 
$$ D(w)_{ii} = \sum_{(i,j) \in E} w_{ij}$$
while the adjacency matrix is given by 
$$ A(w)_{ij} = \begin{cases}
    w_{ij} \qquad &\mbox{if}~(i,j) \in E \\
    0 \qquad &\mbox{otherwise.}
\end{cases}$$
This leads to a Laplacian matrix $L(w) = D(w) - A(w)$ whose eigenvalues depend continuously on the weights. The smallest eigenvalue remains unchanged $\lambda_1(L(w)) = 0$, and since the graph is connected, the second eigenvalue $\lambda_2(L(w))$ is positive. There is a natural scaling symmetry: if we increase all the weights by a factor of 2, then all the eigenvalues of $L(w)$ increase by a factor of 2. Accounting for this scaling, we can now ask whether there is a way of increasing the algebraic connectivity by changing the weights; this amounts to studying
$$ \max\left\{ \lambda_2(L(w)): w_{ij} \geq 0, \, \sum_{i,j=1}^{n} w_{ij} = 2|E| \right\}.$$
The definition of $A(w)$ allows for the weight of an edge to be set to zero (this amounts to erasing the edge), however, it is not possible to create a new edge that was not there before, one can only put weights on the edges of the graph.
This and related notions have been intensively studied, we refer to 
Boyd-Diaconis-Parrilo-Xiao \cite{boyd-diaconis-parrilo-xiao}, G{\"o}ring-Helmberg-Wappler \cite{goering-helmberg-wappler}, G{\"o}ring-Helmberg-Reiss \cite{goering-helmberg-reiss, goering-helmberg-reiss-spectralwidth}, Sun-Boyd-Xiao-Diaconis \cite{sun-boyd-xiao-diaconis} and references therein. We note that this problem has a continuous analogue that has been studied for a much longer time; for the problem of optimizing eigenvalues within a conformal class of metrics on a manifold, we refer to the 1970 paper of Hersch \cite{hersch} (see also El Soufi-Ilias \cite{el}). This connection motivates the terminology of `conformally rigid' since two conformally equivalent metrics on a manifold may be considered the natural analogue of two graphs with the same edge structure but different edge weights.  In the discrete setting, one can think of this as the problem of optimizing the constant in a discrete Wirtinger inequality; we refer to a paper by Fan-Taussky-Todd \cite{fan} for the special case of the cycle graph $C_n$. Additional ways to think about the notion of conformal rigidity and more connections to the literature are given in \cite{steinthomas}.

\subsection{Conformal Rigidity.} 
If we solve the optimization problem $\lambda_2(L(w)) \rightarrow \max$, one naturally expects no particular structure in the optimal weights $w$: they are the solution of a nontrivial optimization problem that depends in a complicated way on the graph. However, it is also to be expected that some graphs are so structured that changing the weights is actually not going to increase $\lambda_2(L(w))$, 
the unweighted graph is the most `algebraically connected' representative in its class.
More precisely, we say that a graph $G=(V,E)$ is \textbf{lower conformally rigid} if, for all choices of non-negative weights normalized to $\sum_{i,j=1}^{n} w_{ij} = 2|E|$
$$ \lambda_2(L(w)) \leq \lambda_2(L(\mathbf{1})),$$
where $\mathbf{1}$ is choosing the weights $w_{ij} \equiv 1$. Analogously, we say that a graph is \textbf{upper conformally rigid} if, again for all admissible choices of weights
$$ \lambda_n(L(w)) \geq \lambda_n(L(\mathbf{1})).$$
A graph is \textbf{conformally rigid} if it is both lower conformally rigid and upper conformally rigid. This terminology slightly extends the one introduced in \cite{steinthomas} where only `conformally rigid' was defined. We found this distinction, corresponding to the two non-trivial endpoints of the spectrum of $L$, to be helpful.

\subsection{Known Results.}
The paper \cite{steinthomas} contains a number of results concerning the structure of conformally rigid graphs. We  summarize those results, and explain how our new contributions complement and extend them.  The existing results describe several classes of graphs that are conformally rigid, summarized in Fig. \ref{fig:summary}.
\begin{center}
    \begin{figure}[h!]
        \centering
    \begin{tikzpicture}
    \node at (0,0) {distance-regular};
        \node at (-2,1) {distance-transitive};
    \node at (2,1) {strongly regular};
    \draw[->] (2,0.8) -- (1.3, 0.2);
        \draw[->] (-2,0.8) -- (-1.3, 0.2);
    \draw [] (-1.75,-1.2) -- (1.8,-1.2) -- (1.8,-0.8) -- (-1.75, -0.8) -- (-1.75, -1.2);
    \node at (0,-1) {\textsc{ conformally rigid}};
      \draw[->] (0,-0.3) -- (0,-0.7);
      \node at (-3,0) {arc-transitive};
          \draw[->] (-2,0.8) -- (-2.5, 0.2);
            \node at (-3,-2) {vertex- and edge-transitive};
                \draw[->] (-3,-0.2) -- (-3, -1.7);
                    \draw[->] (-0.8,-2) -- (-0.2, -2);
                        \draw[->] (0.5, -1.7) -- (0.5, -1.3);
           \node at (1,-2) {edge-transitive};
        \node at (4, 0) {Cayley};
        \draw[->, dashed] (3.8, -0.2) -- (2,-0.8);
        \node at (3.8, -0.6) {\tiny \cite[Theorem 2.3]{steinthomas}};
        \node at (-0.7, -1.5) {\tiny \cite[Proposition 2.1]{steinthomas}};
      \node at (-1.1, -0.5) {\tiny  \cite[Theorem 2.2]{steinthomas}};
    \end{tikzpicture}
        \caption{Summary of the main results from \cite{steinthomas}.}
        \label{fig:summary}
    \end{figure}
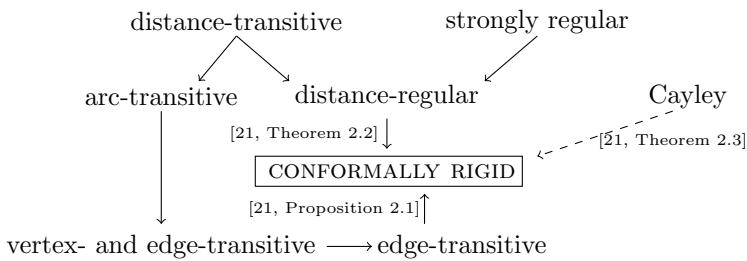
\end{center}
\vspace{-15pt}
Sometimes conformal rigidity is the consequence of a fairly simple underlying structure such as \textit{edge-transitivity}. A graph is edge-transitive if for any two edges there exists a graph automorphism mapping one edge to the other. This simple criterion accounts for many conformally rigid graphs, including cycles, complete graphs, complete bipartite graphs and many others. Among the examples we know of, the smallest conformally rigid graph that is not also edge-transitive, is the {\em Hoffman graph} on 16 vertices (Figure~\ref{fig:conformally rigid}). In \S~4, we develop a more advanced theory for the conformal rigidity of vertex-transitive graphs.\\

Distance-regular graphs are highly structured graphs from Algebraic Graph Theory. They are
regular graphs with the property that for any two vertices $a, b \in V$, the number $\# \left\{v \in V: d(a,v) = i \wedge d(b,v) = j \right\}$ depends only on $i,j$ and $d(a,b)$, and not on the vertices $a$ and $b$ themselves. Theorem 2.2 in \cite{steinthomas} shows that all distance-regular graphs are conformally rigid; this is done by appealing to spectral embeddings. In \S~2 we build a much more comprehensive theory that clarifies the roles of spectral embeddings.

\begin{figure}[h!]
        \begin{tikzpicture}
      \node at (0,0) {  \includegraphics[width=0.25\textwidth]{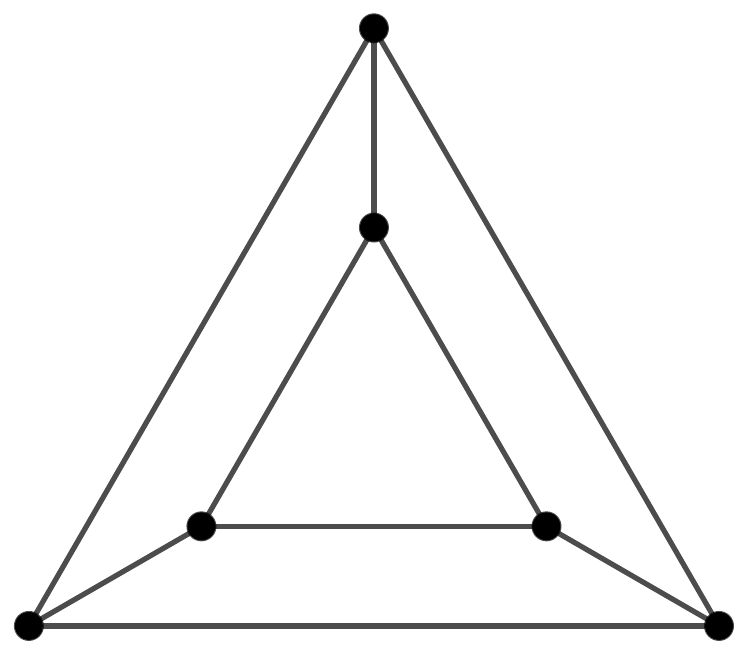}};
       \node at (6,0) {  \includegraphics[width=0.25\textwidth]{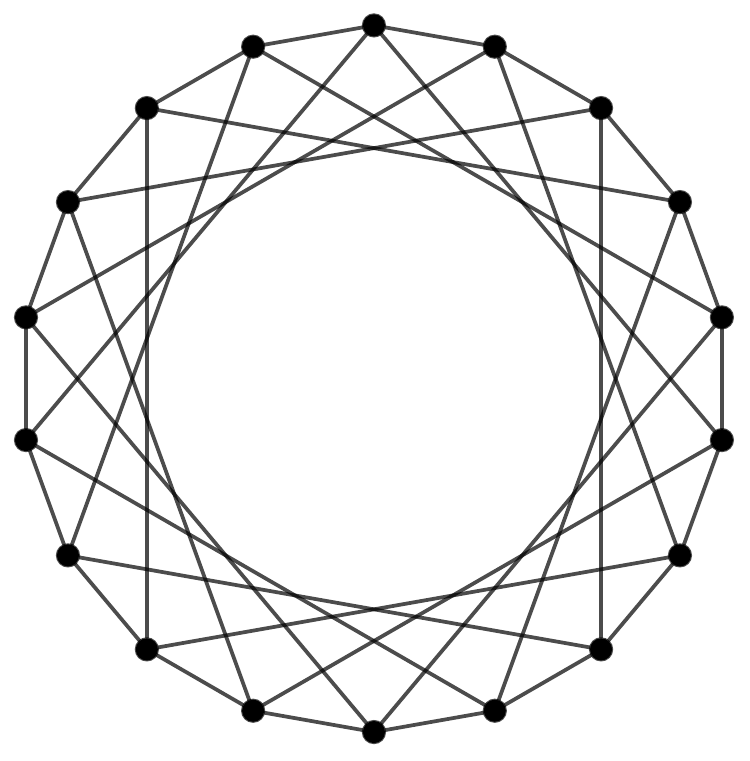}};
       \end{tikzpicture}
        \caption{Left: The triangular prism graph is  Cayley and not conformally rigid. Right: a Cayley graph on $\mathbb{Z}_{18}$ (generated by $S=\left\{-5,-1,1,5\right\}$) that is conformally rigid and \textit{not} edge-transitive.}
        \label{fig:cayley}
    \end{figure}

Cayley graphs can, but need not be, conformally rigid (see Fig. \ref{fig:cayley}). Theorem 2.3 \cite{steinthomas} presents a sufficient criterion for when a Cayley graph is conformally rigid and asks whether it is also necessary. 
Our new Theorem~\ref{thm:necc and suff condn cayley} provides a necessary and sufficient for the conformal rigidity of an abelian Cayley graph.

\subsection{Main Results} 
This section gives an overview of our main results.
Our central tool is the notion of a \textit{spectral embedding} (introduced in \S 2) of the graph. It was established in \cite{steinthomas} that conformal rigidity is equivalent to the existence of a suitable \textit{edge-isometric embedding}, meaning that all edges of the graph in the embedding have a fixed constant length. In \S 2 we use spectral embeddings to study the behavior of conformal rigidity under Cartesian products (Theorem~\ref{thm:cartesian products}).

\S 3 is dedicated to \textit{1-walk-regular graphs} which we show can be characterized by the edge-isometry of their \textit{canonical embeddings} (Theorem~\ref{thm:1-walk regular}).  As a  consequence, they are conformally rigid (a result that was also pointed out to us by Martin Winter). Many conformally rigid graphs that were not explained by the results in \cite{steinthomas} turn out to be walk-regular; see \S 3. 

\S 4 is dedicated to the study of \textit{symmetrized embeddings}. Using this language, we can give a new proof that edge-transitive graphs are conformally rigid (Lemma~\ref{lem:edge-transitive}). However, the framework is more robust and allows us to treat vertex-transitive graphs. If a graph is vertex-transitive with respect to a subgroup of its automorphism group, then conformal rigidity can be certified
by the existence of a suitable eigenvector whose symmetrization is proportional to the constant  vector $\mathbf{1}$ (Corollary~\ref{cor:generaleigenvectorcriterion}). This allows us to confirm the conformal rigidity of many classes by reducing it to a problem of finding the right eigenvector. Corollary~\ref{cor:cr_eq_formulation} provides a more general condition that does not rely on a specific eigenvector.

\S 5 applies the tools from \S 4 to \textit{Cayley graphs}. We  obtain a necessary and sufficient condition for a Cayley graph to be conformally rigid (Theorem~\ref{thm:necc and suff condn cayley}). This result, in particular, provides a nearly complementary converse to a sufficient criterion proved in \cite{steinthomas} using a very different setup.  We  construct an infinite family of conformally rigid circulants, answering a question raised in \cite{steinthomas}.

In \S~6 we interpret symmetrized embeddings in terms of the semidefinite program that underlies conformal rigidity. This provides another criterion for conformal rigidity in the presence of vertex-transitivity (Theorem~\ref{thm:twoedge}).

It is worth emphasizing that when it comes to conformal rigidity of graphs, there are currently two different types of results. The first type of result shows that a graph property (e.g. edge-transitivity, 1-walk regular, distance-regular) implies conformal rigidity. The other type of results show that in the presence of  certain properties (e.g. vertex-transitivity, being Cayley),  conformal rigidity can be {\em certified} by the existence of particular objects such as 
a suitable embedding or eigenvector. However, these certificates do not \textit{explain} the conformal rigidity of the graph. This is the case for many Cayley graphs $\textup{Cay}(\Gamma,S)$ where we can now certify the desired property but it is less clear what property of the group $\Gamma$ or the generator set $S$ has led to the emergence of the property in the first place.

\subsection{Exceptional Examples.} When trying to understand conformally rigid graphs, the authors have found it tremendously helpful to keep track of exceptional examples whose conformal rigidity is not explained by the existing theory. The paper \cite{steinthomas} described an exceptional list of nine conformally rigid graphs. The theory developed in this paper  \textit{explains} two of the nine graphs (they are walk-regular), and  \textit{certifies} an additional five (in the sense that their conformal rigidity is equivalent to the existence of a certain object that we can identify). The remaining two examples are currently `unexplained' in the sense that their conformal rigidity does not follow from a structural theorem. In these two cases, the certificates still exist but we do not understand whether the existence of the certificate is `necessary' or `accidental'. The original list of exceptional examples from \cite{steinthomas} together with their current status is as follows

    \begin{enumerate} 
        \item the Hoffman graph on 16 vertices (1-walk regular, see Corollary \ref{1walk})
        \item the crossing number graph 6B on 20 vertices (remains unexplained)
         \item the Haar graph 565 on 20 vertices ( vertex-transitive and has 2 edge orbits, see Theorem~\ref{thm:twoedge})
         \item the distance-2 graph of the Klein graph on 24 vertices (1-walk regular, see Corollary \ref{1walk})
        \item the $(7,1)-$bipartite $(0,2)-$graph on $n=48$ vertices (remains unexplained)
        \item the $(7,2)-$bipartite $(0,2)-$graph on $n=48$ vertices
             (vertex-transitive and has 2 edge orbits, see Theorem~\ref{thm:twoedge})
      \item the $(20,8)$ accordion graph on 40 vertices (vertex-transitive and has 2 edge orbits, see Theorem~\ref{thm:twoedge})
      \item the non-Cayley vertex-transitive graph (24,23) on $n=24$ vertices (vertex-transitive and has 2 edge orbits, see Theorem~\ref{thm:twoedge})
      \item the 120-Klein graph, discovered by Eric W. Weisstein (vertex-transitive and has 2 edge orbits, see Theorem~\ref{thm:twoedge}).
    \end{enumerate}

We expanded this line of reasoning and tested all graphs in the \textsc{House of Graphs} database \cite{house} that are not edge-transitive (edge-transitive graphs are always conformally rigid). This led to a list of 53 conformally rigid graphs most of which are conformally rigid for reasons explained by the results in this paper. A majority of them are conformally rigid because they are $1$-walk-regular (see Corollary \ref{1walk}), or vertex-transitive with two edge orbits and possess a simple eigenvector certificate (see Corollary \ref{thm:twoedge}). Some are Cayley graphs for which we establish a necessary and sufficient criterion in \S 5. 

The purpose of this section is to list \textbf{five} exceptional graphs that are conformally rigid for reasons not explained by any existing theory. In particular, they are not edge-transitive, not 1-walk regular and not vertex-transitive. It is a curious philosophical question whether these `sporadic' examples indicate  the existence of `missing' Theorems (as with most of the nine exceptional examples listed in \cite{steinthomas}) or whether some of these graphs are conformally rigid simply because they are.
 
  \begin{center}
     \begin{figure}[h!]
     \begin{tikzpicture}
\node at (0,0) {\includegraphics[width=0.25\textwidth]{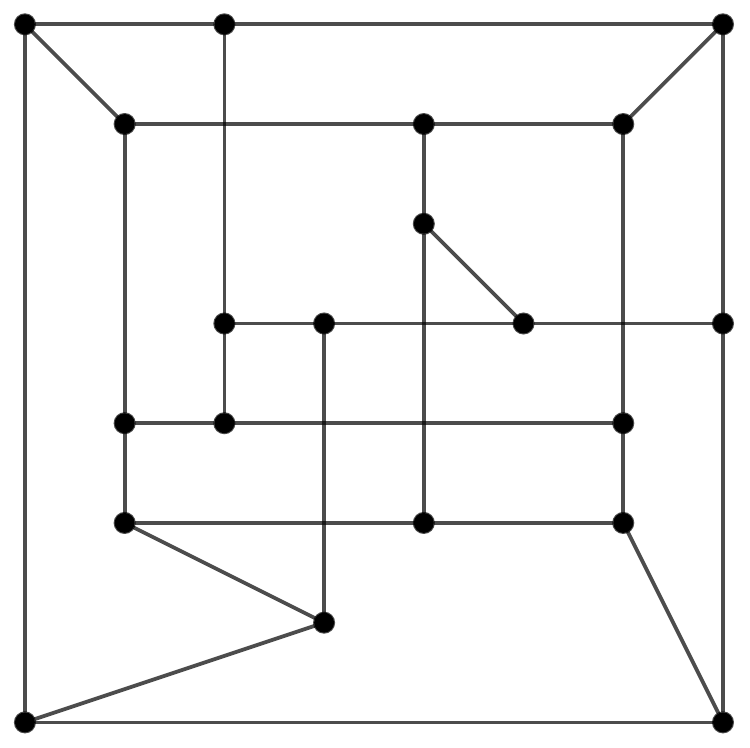}};
\node at (4,0) {\includegraphics[width=0.25\textwidth]{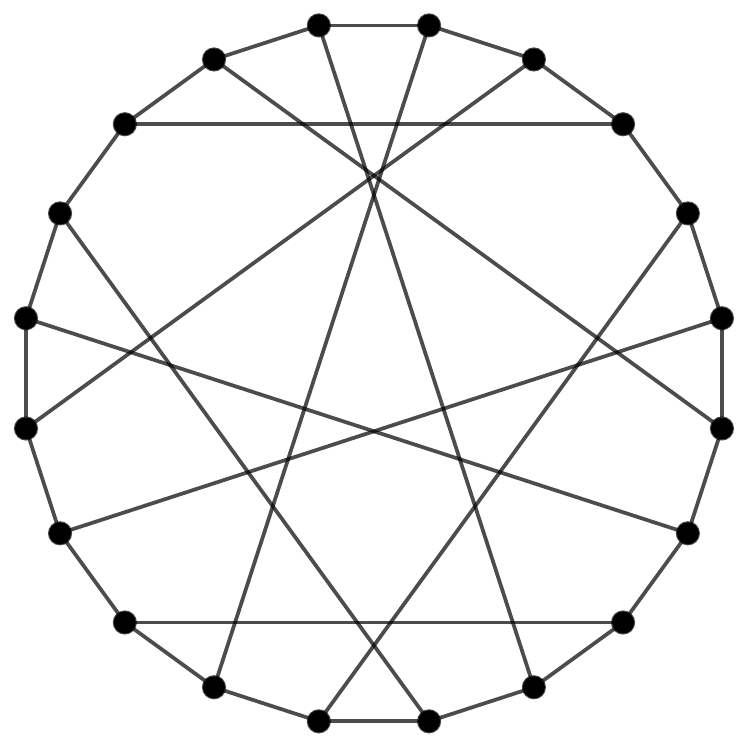}};
\node at (8,0) {\includegraphics[width=0.28\textwidth]{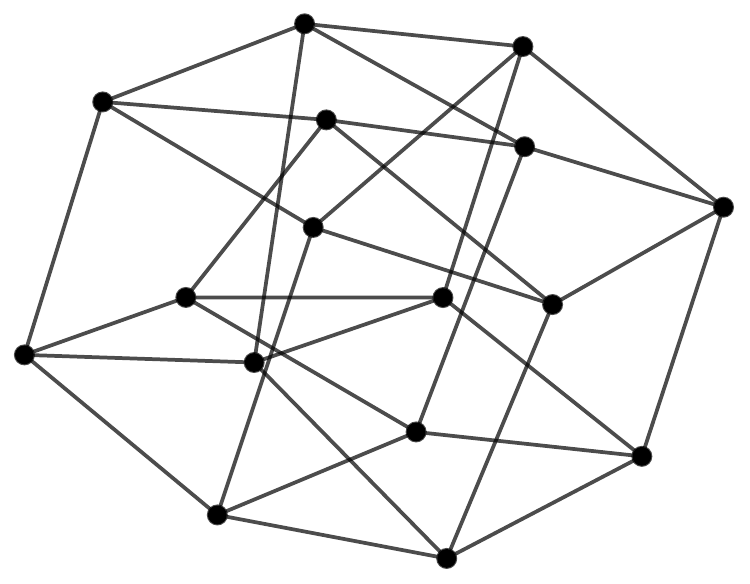}};
\node at (1.5,-4) {\includegraphics[width=0.3\textwidth]{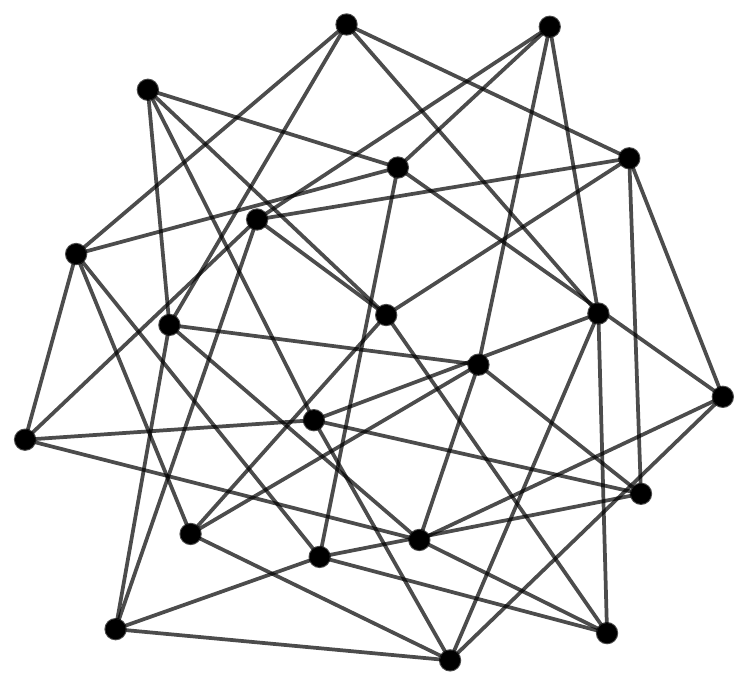}};
\node at (6.5,-4) {\includegraphics[width=0.3\textwidth]{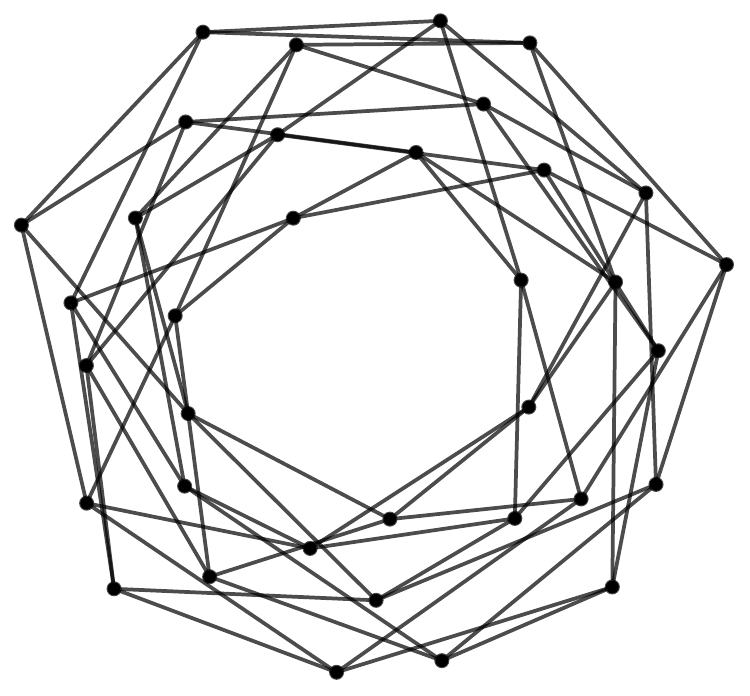}};
\end{tikzpicture}
\caption{Five exceptional graphs that are conformally rigid. Top row. Left: CrossingNumberGraph6B. Middle: the (10,3)-incidence graph 3, Right: \textsc{HoG identity 52594}. Bottom row. Left: \textsc{HoG identity 50405}, right: \textsc{HoG identity 52508}.  }\label{fig:exceptional graphs}
     \end{figure}
 \end{center}

The first example is \textsc{CrossingNumberGraph6B} in Mathematica (HoG identity 1004). This 3-regular graph on $20$ vertices and 30 edges appeared in the 2009 work of Pegg-Exoo \cite{pegg} about graphs minimizing crossing numbers. It already appeared in the list in \cite{steinthomas} and its conformal rigidity remains unexplained. The second example is the \textsc{(10,3)-incidence graph 3} in Mathematica (\textsc{HoG identity 33538}), another 3-regular graph on 20 vertices and 30 edges that is \textit{isospectral} with the first example. It is also remarkable for attaining the largest algebraic connectivity given its diameter, we refer to a recent paper by Exoo-Kolokolnikov-Janssen-Salamon \cite{exoo}. The third example is a 4-regular graph on 16 vertices appearing in House of Graphs as \textsc{HoG identity 52594} and also, rather recently, in a 2024 paper of Goedgebeur-Jooken-Lo-Seamone-Zamfirescu \cite{goed}.  The fourth example is the unitary polarity graph on the projective plane PG(2,4), it has 21 vertices and 48 edges (\textsc{HoG identity 50405}). The final example is a 4-regular graph on 35 vertices and 70 edges that appeared in recent work of Jajcay-Jooken-Porupsanszki \cite{jacjay} (\textsc{HoG identity 52508}). It has the remarkable property that each vertex is contained in exactly 30 pairwise distinct shortest cycles.

\bigskip 
\textbf{Acknowledgments.} We are grateful to Martin Winter for sharing his argument that 1-walk regular graphs are conformally rigid and to Primoz Potocnik for help with the graphs appearing in \cite{primosz}.

\section{Spectral Embeddings}
\subsection{Setup.} Conformal rigidity in a graph can be tested and certified using {\em semidefinite programming}, a branch of convex optimization. For a full development of this connection, see \cite{steinthomas}. In this paper, we will rely on an  interpretation of the semidefinite programming certificate for conformal rigidity in the language of spectral embeddings of graphs. In this section, we will first define the notion of a spectral embedding, give some examples, and then use spectral embeddings to study the behavior of conformal rigidity under Cartesian products.

\begin{definition} \label{def:embedding}
Let  $\lambda >0$  be an eigenvalue of the Laplacian 
$L$ of $G=(V,E)$, of multiplicity $m$, and let $\mathcal{E}_\lambda \cong \RR^{m}$ be the eigenspace of $\lambda$ 
and $P \in \RR^{n \times k}$ a matrix whose 
columns lie in $\mathcal{E}_\lambda$. The collection of vectors 
$\mathcal{P} = \{ 
p_i, \,\,i=1, \ldots, n\} \subset \RR^{k}$, such that $p_1^\top, \ldots p_n^\top$ are the rows of $P$,  
is a {\bf spectral embedding} of $G$ on $\mathcal{E}_\lambda$. 
\end{definition}

The configuration $\mathcal{P}$ embeds $G$ in $\RR^{k}$ by assigning the vector $p_i$ to vertex $v_i$, and connecting  $p_i$ and $p_j$ by an edge for all $(i,j) \in E$. The embedding is {\em centered} in the sense that $\sum p_i = 0$ since the columns of $P$ are 
orthogonal to ${\bf 1}$.  
Spectral embeddings were defined by Hall \cite{hall} and several variants have been studied. 

\begin{definition} \label{def:edge iso embedding}
    An embedding $\mathcal{P}$ of $G$ on $\mathcal{E}_\lambda$ is said to be 
    {\bf edge-isometric} if there exists a constant $c > 0$ such that 
    $\|p_i - p_j \|=c$ for all $(i,j) \in E$. 
\end{definition}

Since $G$ is connected, the eigenspace $\mathcal{E}_0$ is spanned by ${\bf 1}$ and hence an embedding of $G$ using any multiple of ${\bf 1}$ will have $\|p_i - p_j \|=0$ for all $(i,j) \in E$. Hence, Definition~\ref{def:edge iso embedding} requires $\lambda > 0$ and $c > 0$. 

Edge-isometric {\em spectral} embeddings are rather special, and not the same as an edge-isometric embedding of the graph in some $\RR^k \subset \RR^m$, where $m= \dim( \mathcal{E}_\lambda)$.  

\begin{example}
Any bipartite graph $G$ with vertex partition $A \cup B$ has an edge-isometric embedding in $\RR$ obtained by sending all the vertices in $A$ to $1$ and all the vertices in $B$ to $-1$. Consider a path on $4$ vertices 
\begin{center}
    \begin{tikzpicture}
    \node[draw, circle, fill=black, inner sep=2pt, label=above:$1$] (v1) at (0,0) {};
    \node[draw, circle, fill=black, inner sep=2pt, label=above:$2$] (v2) at (2,0) {};
    \node[draw, circle, fill=black, inner sep=2pt, label=above:$3$] (v3) at (4,0) {};
    \node[draw, circle, fill=black, inner sep=2pt, label=above:$4$] (v4) at (6,0) {};

    \draw (v1) -- (v2) -- (v3) -- (v4);
\end{tikzpicture}
\end{center}
which is bipartite, and 
whose Laplacian is 
$$L=\begin{bmatrix} 
1 & -1 & 0 & 0 \\ 
-1 & 2 & -1 & 0 \\
0 & -1 & 2 & -1 \\
0 & 0 & -1 & 1 \end{bmatrix}$$
with eigenvalues $0, 2-\sqrt{2}, 2, 2+\sqrt{2}$. 
An eigenvector of the largest eigenvalue $2+\sqrt{2}$ is 
$\varphi = (-1,1+\sqrt{2},-1-\sqrt{2},1)$, and hence the spectral embedding of $G$ on the last eigenspace has vertices at the coordinates of $\varphi$:
\begin{center}
\begin{tikzpicture}
    \node[draw, circle, fill=black, inner sep=2pt, label=above:3] (v1) at (-1.414-1,0) {};
    \node[draw, circle, fill=black, inner sep=2pt, label=above:1] (v2) at (-1,0) {};
    \node[draw, circle, fill=black, inner sep=2pt, label=above:4] (v3) at (1,0) {};
    \node[draw, circle, fill=black, inner sep=2pt, label=above:2] (v4) at (1+1.414,0) {};

    \draw (v1) -- (v2) -- (v3) -- (v4);
\end{tikzpicture}
\end{center}
The edges $(1,2)$ and $(3,4)$ have the same length but $(2,3)$ is longer which means that no spectral embedding in this eigenspace can be edge-isometric. \qed
\end{example}

The following modification of a result from \cite{steinthomas} will be our main tool. It is an adaptation of an observation in 
\cite{goering-helmberg-wappler} and \cite{sun-boyd-xiao-diaconis} in the context of conformal rigidity. 

\begin{proposition}\cite[Proposition 4.3]{steinthomas} \label{prop:embedding certificates}
A graph $G=([n],E)$ is lower conformally rigid if and 
only if it has an edge-isometric embedding on $\mathcal{E}_{\lambda_2}$. It is upper conformally rigid if and only if it has an edge-isometric embedding on $\mathcal{E}_{\lambda_n}$.
\end{proposition}

We introduce a few more definitions of embeddings needed in this paper.

\begin{definition}
    Let $\mathcal{P}=\{p_1, \ldots, p_n\}$ be an embedding of $G$ on $\mathcal{E}_\lambda$ where $\lambda > 0$. 
    \begin{enumerate}
        \item $\mathcal{P}$ is a {\bf spherical embedding} if all vectors $p_i$ have the same norm. 
        \item $\mathcal{P}$ is the {\bf canonical embedding} if the columns of $P$ form an orthonormal basis of $\mathcal{E}_\lambda$. 
        \item $\mathcal{P}$ is a {\bf symmetrized embedding} if the columns of $P$ arise as follows: pick a nonzero vector $\varphi \in \mathcal{E}_\lambda$ and let the columns of $P$ be the elements in the orbit of $\varphi$ under a subgroup $H$ of $\textup{Aut}(G)$, the automorphism group of $G$. 
    \end{enumerate}
\end{definition}

  If $P$ and $P'$ are two matrices whose columns form orthonormal bases of $\mathcal{E}_\lambda$, we have that $P' = PU$ for an orthogonal matrix $U$. Therefore, $\mathcal{P}' = \{ U^\top p_1, \ldots, U^\top p_n\}$ is obtained by an orthogonal transformation of $\mathcal{P}$. This preserves angles and lengths and hence we call any one of these embeddings `the' canonical embedding of $G$ on $\mathcal{E}_\lambda$. 
A symmetrized embedding makes sense since every eigenspace of $L$ is $\textup{Aut}(G)$-invariant. 
 Every spectral embedding is a linear image of the canonical embedding of $G$ on that eigenspace. Indeed, if the columns of $U$ form an orthonormal basis of $\mathcal{E}_\lambda$ and $P$ is any matrix with columns in $\mathcal{E}_\lambda$, then $P = U B$ for some matrix $B$. Denoting the rows of $U$ by $u_1^\top, \ldots, u_n^\top$, we have  
$p_i = B^\top u_i$ for all $i$.

\subsection{An application: conformal rigidity under Cartesian products}
Given two conformally rigid graphs $G$ and $H$ as well as their associated spectral embeddings, we can (under some conditions) give an explicit algebraic construction of a spectral embedding for $G \square H$, the Cartesian product of $G$ and $H$. 

The Cartesian product of $G$ and $H$ has vertex set $V(G) \times V(H)$, and two elements $(g_1, h_1), (g_2, h_2) \in V(G) \times V(H)$ form an edge in $G \square H$ if either $g_1 = g_2$ and $(h_1, h_2) \in E(H)$ or $(g_1, g_2) \in E(G)$ and $h_1 = h_2$. 
Given the Laplacian matrices $L_{G}$ and $L_{H}$ of $G$ and $H$, the Laplacian matrix of  \( G \square H \) is  
\[
L_{G \square H} = L_G \otimes I_H + I_G \otimes L_H,
\]
where $\otimes$ denotes the tensor product and $I_H, I_G$ denote the identity matrices of sizes $|V(G)|$ and $|V(H)|$. If \( \lambda_i^G \) are the eigenvalues of \( L_G \) and \( \lambda_j^H \) are the eigenvalues of \( L_H \), then the eigenvalues of \( L_{G \square H} \) are given by
\[
\lambda_{i,j}^{G \square H} = \lambda_i^G + \lambda_j^H, \quad \forall i, j.
\]
If \( v_i^G \) and \( v_j^H \) are the corresponding eigenvectors of \( L_G \) and \( L_H \),  then the eigenvectors of \( L_{G \square H} \) are given by the Kronecker products $v_{i,j}^{G \square H} = v_i^G \otimes v_j^H.$
It is now possible to construct spectral edge-isometric embeddings in several cases.

\begin{lemma}
If $G$ and $H$ have edge-isometric embeddings on  their $\mathcal{E}_{\lambda_2}$'s and have the same positive algebraic connectivity, i.e., $\lambda_2(G)=\lambda_2(H)$, then $G \square H$ has an edge-isometric embedding on its $\mathcal{E}_{\lambda_2}$.
\end{lemma}
\begin{proof}
Let the rows of $U_G$ and $U_H$ be the edge-isometric spectral embeddings of $G$ and $H$, which implies that their columns are eigenvectors associated to $\lambda_2$ in $L_G$ and $L_H$ respectively. Since the eigenvector associated to $\lambda_1=0$ is always the all-ones vector, we have that $U_G \otimes {\bf 1}_{|H|}$ and ${\bf 1}_{|G|} \otimes U_H$ have columns that are eigenvectors of $\lambda_2$ in $G \square H$. Consider the matrix $U=[U_G \otimes {\bf 1}_{|H|}, {\bf 1}_{|G|} \otimes U_H]$. The row $(i,j)$ of $U$ is simply the concatenation of the $i$-th row of $U_G$ and the $j$-th row of $U_H$. So the embedding of $G \square H$ given by the rows of $U$ is just the cartesian product of the original embeddings of each graph, and is therefore still edge-isometric.
\end{proof}

If we normalize a spherical edge-isometric embedding of $G$ on $\mathcal{E}_{\lambda_{\max}}$ so that 
all edge lengths are $1$, then the radius  of the embedding is $\sqrt{\delta_G/(2 \lambda_{\max}(G))}$ where $\delta_G$ is the average degree of $G$ \cite[Remark 4.4]{steinthomas}.

\begin{lemma}
If $G$ and $H$ have spherical edge-isometric embeddings on $\mathcal{E}_{\lambda_{\max}}$ and 
$$\frac{\delta_G}{2 \lambda_{\max}(G)}=\frac{\delta_H}{2 \lambda_{\max}(H)},$$ 
then $G \square H$ has a {spherical}  edge-isometric embedding on its $\mathcal{E}_{\lambda_{\max}}$.
\end{lemma}
\begin{proof}
Note that $\lambda_{\max}(G\square H) = \lambda_{\max}(G) + \lambda_{\max}(H)$.
Let the rows of $U_G$ and $U_H$ be spherical edge-isometric embeddings of $G$ and $H$ on their respective $\mathcal{E}_{\lambda_{\max}}$. We can assume that the edge lengths are $1$ in both embeddings and so both embeddings have the same radius $r$ from the {hypothesis of the lemma}.
The columns of $U=U_G \otimes U_H$ are eigenvectors of $\lambda_{\max}$ in $G \square H$. Moreover, the row $(i,j)$ of $U$ is the tensor product of the $i$-th row of $U_G$ and the $j$-th row of $U_H$, which implies $$\|u_j \otimes v_i - u_j \otimes v_{i'}\|=\|u_j\| \|v_i-v_{i'}\|=r $$
 for any edge $(j,i) \sim (j,i')$ and similarly for edges $(j,i) \sim (j',i)$. This gives us an edge-isometric spherical embedding of $G \square H$ as claimed.
\end{proof}

We combine the above results to get the following theorem.

\begin{theorem} \label{thm:cartesian products}
If $G$ and $H$ are conformally rigid, $\lambda_2(G)=\lambda_2(H)$, both have a spherical embedding on  $\mathcal{E}_{\lambda_{\max}}$,   and 
$$\frac{\delta_G}{2 \lambda_{\max}(G)}=\frac{\delta_H}{2 \lambda_{\max}(H)},$$ 
then $G \square H$ is conformally rigid.
\end{theorem}

The condition of having a spherical edge-isometric embedding on $\mathcal{E}_{\lambda_{\max}}$ is automatically true for $1$-walk regular graphs (Theorem~\ref{thm:1-walk regular}), conformally rigid vertex-transitive graphs (Lemma~\ref{lem:vertex-transitive}), and regular bipartite graphs. In fact, almost all of our conformally rigid examples have this property and the condition does not seem to be too restrictive in the class of conformally rigid graphs. 
Since the Cartesian product preserves both $\lambda_2$ and $\delta/\lambda_{\max}$, one can use Theorem~\ref{thm:cartesian products} to keep taking products indefinitely to generate infinite families of conformally rigid graphs.

\section{$1$-Walk Regular Graphs}
In this section, we showcase an application of the canonical embedding  of a graph. 
In a private communication, Martin Winter shared a proof that 
$1$-walk regular graphs are conformally rigid. His argument relies on 
 Proposition~\ref{prop:embedding certificates} and canonical embeddings. We strenghten his argument  to show that 
 a graph is $1$-walk regular if and only if its canonical embedding on {\em every} eigenspace is 
 both spherical and edge-isometric (Theorem~\ref{thm:1-walk regular}).   
 
Suppose $G$ is a regular graph (of degree $d$) with adjacency matrix $A$. The spectral information of $A$ and $L$ are parallel; $\lambda$ is an eigenvalue of $A$ if and only if $d-\lambda$ is an eigenvalue of $L$, and both eigenvalues have the same eigenspace.

\begin{definition}
\begin{enumerate}
    \item A regular graph $G$ is {\bf $0$-walk regular} if for any vertex $v$ and any positive integer $\ell$, the number of length $\ell$ walks from $v$ to $v$ in $G$ does not depend on $v$. 

    \item $G$ is {\bf $k$-walk regular} if for any pair of vertices $v,v'$ such that $\textup{dist}(v,v') \leq k$, and any positive integer $\ell$, the number of length $\ell$ walks from $v$ to $v'$ depend only on $\textup{dist}(v,v')$ and $\ell$ and not on the choice of the vertex pair. 

    \item In particular, $G$ is {\bf $1$-walk regular} if it is $0$-walk regular and for every pair of adjacent vertices $v,v'$ and any positive integer $\ell$, the number of length $\ell$ walks from $v$ to $v'$ depend only on $\ell$ and not on the choice of edge. 
\end{enumerate}
\end{definition}

Note that $G$ is $0$-walk regular if and only if $A^s$, the $s$th power of the adjacency matrix $A$, has constant diagonal for all $s$, and $G$ is $1$-walk regular if and only if $A^s$ has a constant diagonal for all $s$ and all entries of $A^s$ in positions $(i,j)$ where $(i,j) \in E$ are the same. Distance-regular graphs are $k$-walk regular for $k$ equal to the graph's diameter. In particular, all distance-regular graphs are $1$-walk regular.

\begin{theorem} \label{thm:1-walk regular}
    A regular graph $G$ is $1$-walk regular if and only if its canonical embedding on {\em every} eigenspace is both spherical and edge-isometric. 
\end{theorem}

\begin{proof}
Let $\mathcal{A} = \textup{span}\{I, A, A^2, A^3, \ldots\}$ be the {\em adjacency algebra} of $G$ \cite[Chapter 2]{biggs}, and let $\mathcal{E}_1, \ldots, \mathcal{E}_m$ be the distinct eigenspaces of $A$. For each $\mathcal{E}_k$, let 
    $\mathcal{U}_k = \{ u_{k1}, u_{k2}, \ldots, u_{kn}\}$ be the canonical embedding of $G$ on $\mathcal{E}_k$. It is known that the orthogonal projectors $U_1U_1^\top, \ldots, U_mU_m^\top$ form a basis of $\mathcal{A}$, as do $I, A, A^2, \ldots, A^{m-1}$. This implies that the elements of one basis can be written uniquely in terms of the other basis:
    \begin{align}
         U_sU_s^\top = \sum_{k=0}^{m-1} \alpha_{sk} A^k,\label{eq:U in terms of A}\\
         A^s = \sum_{k=1}^m \beta_{sk} U_k U_k^\top. \label{eq:A in terms of U}
    \end{align}
    From \eqref{eq:U in terms of A}, we get that 
    \begin{align} 
    \langle u_{si}, u_{sj} \rangle = (U_sU_s^\top)_{ij} = \sum_{k=0}^{m-1} \alpha_{sk} (A^k)_{ij} = \sum_{k=0}^{m-1} \alpha_{sk} w_k(i,j)
    \end{align}
    where $w_k(i,j)$ is the number of length $k$ walks from vertex $i$ to vertex $j$. 
    
    If $G$ is $1$-walk regular, $w_k(i,i)$ and $w_k(i,j)$ are independent of $i \in V$ and $(i,j) \in E$, respectively. Therefore, 
    $\langle u_{si}, u_{si} \rangle$ is independent of $i$ and 
    $\langle u_{si}, u_{sj} \rangle$ is independent of $(i,j) \in E$, which means that $\|u_{si} - u_{sj}\|$ is also independent of $(i,j) \in E$.
    We conclude that the canonical embedding $\mathcal{U}_k$ is both spherical and edge-isometric.
    
    Conversely, suppose the canonical embedding $\mathcal{U}_k$ is spherical and edge-isometric, i.e., $\|u_{ki}\|$ is the same positive number $\mu$ for all $i=1, \ldots, n$ and $\| u_i - u_j \|$ is the same positive number $\nu$ for all $(i,j) \in E$.
    Then,
    \begin{align}
        (A^s)_{ii} =  \sum_{k=1}^m \beta_{sk} (U_k U_k^\top)_{ii} = \sum_{k=1}^m \beta_{sk} \|u_{ki}\|^2 = \mu^2 \sum_{k=1}^m \beta_{sk}
   \end{align}
   which is independent of $i$ and so, the matrix $A^s$ has a constant diagonal. Since  $\mathcal{U}_k$ is edge-isometric, for all $(i,j) \in E$,
   \begin{align}
         \nu^2 = \|u_{ki} - u_{kj}\|^2 = \|u_{ki}\|^2 + \|u_{kj}\|^2 - 2 u_{ki}^\top u_{kj} = 2\mu^2 - 2u_{ki}^\top u_{kj}.
   \end{align}
   This implies that for all $(i,j) \in E$, all $u_{ki}^\top u_{kj}$ are equal to the same constant $\gamma$, and 
   \begin{align}
        (A^s)_{ij} =  \sum_{k=1}^m \beta_{sk} (U_k U_k^\top)_{ij} = \sum_{k=1}^m \beta_{sk} u_{ki}^\top u_{kj} = \gamma \sum_{k=1}^m \beta_{sk}
    \end{align}
    is independent of $(i,j)$. In other words, all entries of $A^s$ in positions $(i,j)$ for $(i,j) \in E$ are the same. Thus we have that $G$ is $1$-walk regular.
\end{proof}

\begin{center}
    \begin{figure}[h!]
     \begin{tikzpicture}
         \node at (0,0) {\includegraphics[width=0.35\textwidth]{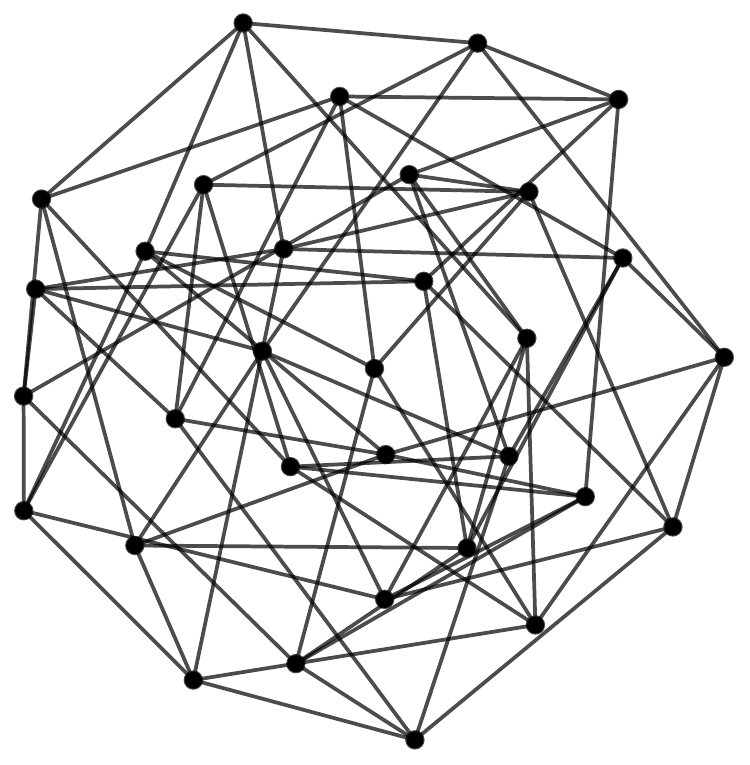}};
         \node at (6,0) {\includegraphics[width=0.35\textwidth]{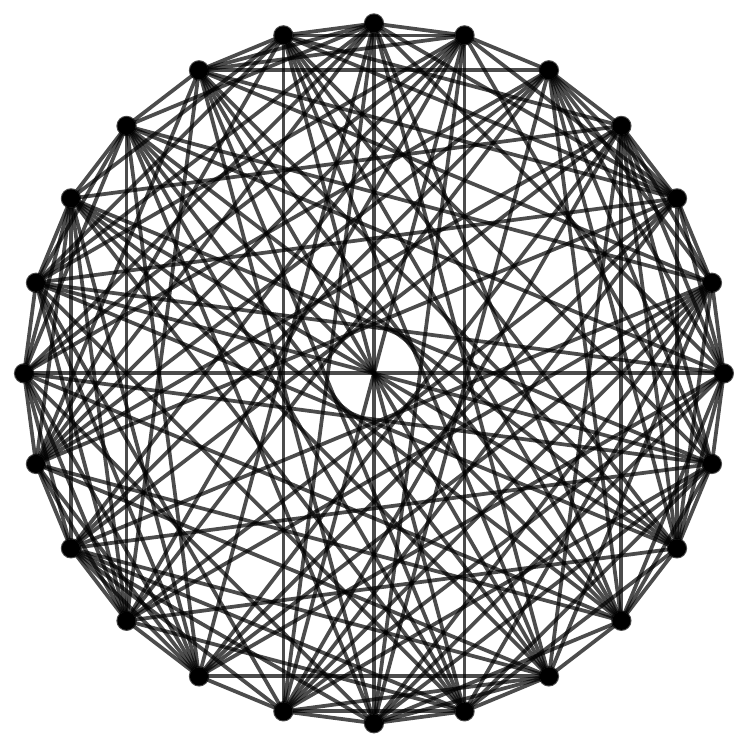}};
     \end{tikzpicture}
     \caption{The `Brussels graph' (\textsc{HoG 50460}, \cite{araujo}) and the \textsc{Klein Distance 2 graph} are 1-walk regular, thus conformally rigid.}
    \end{figure}
\end{center}

As a corollary we get the following result from Winter which subsumes our previous result in \cite{steinthomas} that distance-regular graphs are conformally rigid. 

\begin{corollary}[Martin Winter] \label{1walk}
All $1$-walk regular graphs are conformally rigid.
\end{corollary}

\begin{example}
    Consider the $4$-regular graph 52508 from the House of Graphs database (see Figure~\ref{fig:exceptional graphs}) which has $35$ vertices and $70$ edges. This graph is  not $1$-walk regular but is conformally rigid. In fact, canonical embeddings certify conformal rigidity. 
    Since the graph is not $1$-walk regular, it must have an eigenspace in which the canonical embedding is not edge-isometric. Indeed, the third largest eigenvalue of $L$ is $4-\sqrt{5}$ and it has a one-dimensional eigenspace spanned by 
    \begin{align} 
    \begin{split}
    \varphi = & ( 0,-1,1,2-\sqrt{5},\sqrt{5}-2,0,0,\sqrt{5}-2,-1,1,2-\sqrt{5},0,0,2-\sqrt{5},\sqrt{5}-2,\\ & 1,-1,\sqrt{5}-2,1,2-\sqrt{5},-1,0,0,1,-1,2-\sqrt{5},\sqrt{5}-2,\sqrt{5}-2,2-\sqrt{5},\\&-1,1,-1,2-\sqrt{5},\sqrt{5}-2,1)
    \end{split}
    \end{align}
    In this embedding, edge $(1,2)$ has length $1$ while $(1,5)$ has length $\sqrt{5}-2$. This embedding is also not spherical.
\end{example}

\section{Symmetrized Embeddings}
Suppose $\textup{Aut}(G)$ denotes the automorphism group of $G$. 
A subgroup $\Psi$ of $\textup{Aut}(G)$ can be identified with a collection of permutations $\sigma$ on $[n]$ and hence, the action of $\sigma \in \Psi$ on $x \in \RR^n$ permutes the coordinates of $x$. In particular, $(\sigma x)_i = x_{\sigma(i)}$ for all $i=1,\ldots,n$. 
A $\Psi$-symmetrized spectral embedding of $G$ on $\mathcal{E}_\lambda$, $\lambda > 0$, is obtained by first picking an element $\varphi \in \mathcal{E}_\lambda$ and then constructing a matrix $P$ whose columns are all vectors $\sigma \varphi$ as $\sigma$ varies in $\Psi$. These columns are all in $\mathcal{E}_\lambda$ since $\mathcal{E}_\lambda$ is $\textup{Aut}(G)$-invariant  and hence also $\Psi$-invariant.
 The rows of $P$ form the $\Psi$-symmetrized embedding of $G$ on $\mathcal{E}_\lambda$ which we write as $\mathcal{P}_\Psi$. Keep in mind that if 
\begin{align}
    \mathcal{P}_\Psi = \{p_1, \ldots, p_n\} \textup{ then } p_i = (\varphi_{\sigma(i)})_{\sigma \in \Psi}.
\end{align}

Here is a reason to care about symmetrized spectral embeddings of $G$. 

\begin{lemma} \label{lem:edge-transitive}
    If $G$ is edge-transitive with respect to $\Psi$, then for any $\lambda > 0$, a $\Psi$-symmetrized embedding 
    $\mathcal{P}_\Psi$ on $\mathcal{E}_\lambda$ is edge-isometric. 
\end{lemma}
\begin{proof}
    The vector $p_i$ is the embedding of vertex $v_i$ and its squared norm is 
    \begin{align} \label{eq:norm square vertex}
        \|p_i\|^2 = p_i^\top p_i = \sum_{\sigma \in \Psi} \varphi_{\sigma(i)}^2.
    \end{align}
    The last expression tells us that $\|p_i\|^2$ only depends on the $\Psi$-orbit of $v_i$ and not on the chosen $v_i$. 
        For an edge $(i,j) \in E$, 
    \begin{align} \label{eq:norm square edge}
    \begin{split}
        \|p_i - p_j \|^2 & =p_i^\top p_i + p_j^\top p_j - 2p_i^\top p_j \\
        & = \sum_{\sigma \in \Psi} \varphi_{\sigma(i)}^2 + \sum_{\sigma \in \Psi}  \varphi_{\sigma(j)}^2 - 2 \sum_{\sigma \in \Psi}  \varphi_{\sigma(i)} \varphi_{\sigma(j)}.
     \end{split}   
    \end{align}
    Since $G$ is edge-transitive with respect to $\Psi$, all edges are in a single $\Psi$-orbit and the last sum is independent of the chosen edge $(i,j)$. Further, the vertices of $G$ fall into at most two orbits; if there are two vertex orbits then they form a bipartition of $G$. Either way, the first two sums are either one or two fixed numbers, depending on the graph $G$, and  $\mathcal{P}_\Psi$ is edge-isometric as long as $\|p_i - p_j\| > 0$.

    If $\|p_i - p_j\|=0$ for some $(i,j) \in E$, then 
    $\varphi_{\sigma(i)} = \varphi_{\sigma(j)}$ for all $\sigma \in \Psi$. Since $G$ is edge-transitive, for any edge $(k,\ell) \in E$, there is a $\sigma \in \Psi$ such that $k = \sigma(i)$ and $\ell = \sigma(j)$, or $k = \sigma(j)$ and $\ell = \sigma(i)$.  By the connectivity of $G$ it must be that $\varphi$ is a constant vector which contradicts that $\lambda > 0$.
\end{proof}

Lemma~\ref{lem:edge-transitive} gives a new proof that edge-transitive graphs are conformally rigid, a result that was proved in \cite{steinthomas}. Next we consider vertex-transitive graphs. 

\begin{lemma} \label{lem:vertex-transitive}
    If $G$ is vertex-transitive with respect to $\Psi$, then a symmetrized embedding $\mathcal{P}_\Psi$ on $\mathcal{E}_\lambda$ is spherical.
\end{lemma}

\begin{proof}
    Since all vertices are in the same $\Psi$-orbit,  \eqref{eq:norm square vertex} is the same for all vertices of $G$, and 
    $\|p_i\|^2 \geq \|\varphi\|^2 > 0$.
\end{proof}

Let then $G$ be vertex-transitive with respect to $\Psi$, and we fix a $\Psi$-symmetrized embedding $\mathcal{P}_\Psi$ of $G$ on $\mathcal{E}_\lambda$, where again, $p_i$ is the embedding of vertex $v_i$. For an edge $(i,j) \in E$ consider the equation \eqref{eq:norm square edge}. The first two sums are just the squared norms of $p_i$ and $p_j$ which are equal since all vertices lie in a single $\Psi$-orbit. Therefore, 
\eqref{eq:norm square edge} depends only on the sum
\begin{align}\label{eq:edge quantity}
    \sum_{\sigma \in \Psi}  \varphi_{\sigma(i)} \varphi_{\sigma(j)}
\end{align}
which depends only on the $\Psi$-orbit of $(i,j)$.

\begin{definition} \label{def:phiH}
Let $\mathcal{EO}_\Psi$ denote the collection of edge orbits of $G$ with respect to a subgroup $\Psi$ of $\textup{Aut}(G)$. 
 For $\lambda > 0$ and $\varphi \in \mathcal{E}_\lambda$ with $\|\varphi\|=1$, define the following vector with components indexed by $\mathcal{EO}_\Psi$:
\begin{align} \label{eq:phipsi vector}
    \varphi_\Psi := \left(  \sum_{\sigma \in \Psi}  \varphi_{\sigma(i)} \varphi_{\sigma(j)} \right)_{\mathcal{EO}_\Psi}. 
\end{align}
\end{definition}

\begin{theorem} \label{thm:eigevec certificate}
    Suppose $G$ is vertex-transitive with respect to a subgroup $\Psi$ of $\textup{Aut}(G)$, and $\lambda>0$. Then $G$ has an edge-isometric $\Psi$-symmetrized embedding on $\mathcal{E}_\lambda$ if and only if there is an eigenvector $\varphi \in \mathcal{E}_\lambda$, $\|\varphi\|=1$, for which $\varphi_\Psi$ is a multiple of ${\bf 1}$. This embedding is also spherical when it exists. 
\end{theorem}

\begin{proof}
By our observations above, we have an edge-isometric $\Psi$-symmetrized embedding based on an eigenvector $\varphi  \in \mathcal{E}_\lambda$, if and only if  $\varphi_\Psi$ is a constant vector and $\|p_i-p_j\|^2>0$ for any edge $(i,j)$.  If $\|p_i-p_j\|^2=0$ then $\varphi_{\sigma(i)} = \varphi_{\sigma(j)}$ for all $\sigma \in \Psi$ which means that $\varphi_i = \varphi_j$. Therefore, $\varphi$ 
is constant along connected components of $G$ and hence, $\lambda=0$, a contradiction. 
The last statement of the theorem follows from Lemma~\ref{lem:vertex-transitive}.
\end{proof}

This means that the existence of a constant $\varphi_\Psi$ certifies the existence of a spherical edge-isometric ($\Psi$-symmetrized) embedding of $G$ on $\mathcal{E}_\lambda$. We remark that the condition $\| \varphi \|=1$ is just a convenient normalization, and not integral to the result. 

\begin{corollary} \label{cor:generaleigenvectorcriterion}
    Suppose $G$ is vertex-transitive with respect to a subgroup $\Psi$ of $\textup{Aut}(G)$. Then $G$ is lower conformally rigid if there is a $\varphi \in \mathcal{E}_{\lambda_2}$ for which  $\varphi_\Psi$ is a multiple of ${\bf 1}$. The analogous statement holds for upper conformal rigidity. 
\end{corollary}

\begin{example}
 Take $G$ to be the complement of the Shrikhande graph shown in Figure~\ref{fig:Shrikhande complement}. 
 \begin{figure}[h!]
     \centering
     \includegraphics[width=0.4\linewidth]{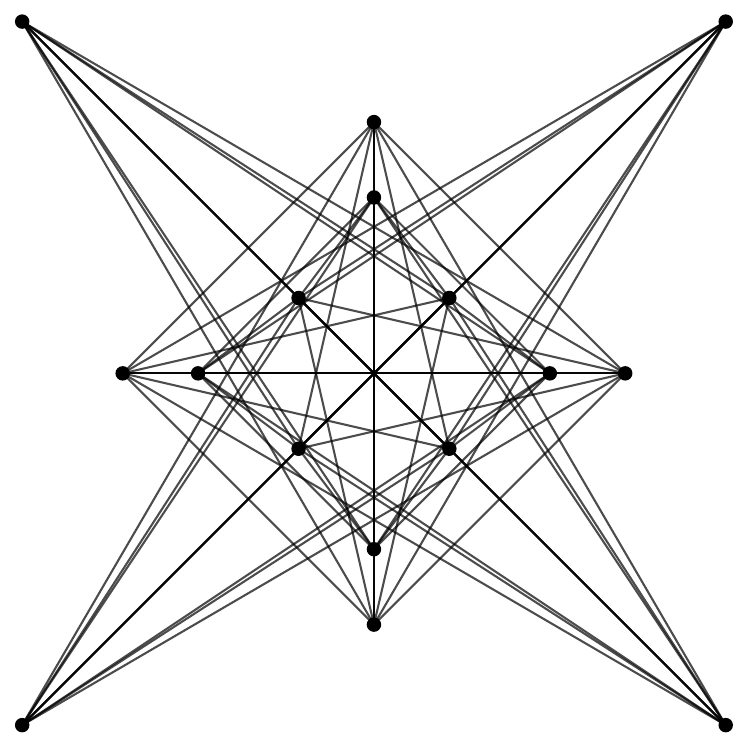}
     \caption{The complement of the Shrikhande graph.}
     \label{fig:Shrikhande complement}
 \end{figure}
 This graph has $16$ vertices, is vertex-transitive with respect to $\textup{Aut}(G)$  but not edge-transitive. Its $72$ edges fall into two orbits with respect to $\textup{Aut}(G)$. The Laplacian $L$ has three distinct eigenvalues: $12, 8, 0$ with multiplicities $6,9,1$ respectively. The graph $G$ is $1$-walk regular and hence conformally rigid. 
 We now show that $G$ has a certificate as in Corollary~\ref{cor:generaleigenvectorcriterion} for $\lambda_2 = 8$. Pick 
 
 {\small{
 \begin{align*} 
 \varphi 
 & =  \frac{3}{7} \left(-\sqrt{\frac{5}{3}}-\sqrt{\frac{11}{15}}\right) 
 \underbrace{\left(0,0,\frac{1}{\sqrt{6}},-\frac{1}{\sqrt{6}},-\frac{1}{\sqrt{6}},\frac{1}{\sqrt{6}},0,-\frac{1}{\sqrt{6}},0,0,0,0,0,0,0,\frac{1}{\sqrt{6}}\right)}_{\varphi^1} \\
 & +  \underbrace{\left(\sqrt{\frac{2}{5}},-\frac{1}{\sqrt{10}},-\frac{1}{\sqrt{10}},0,-\frac{1}{\sqrt{10}},-\frac{1}{\sqrt{10}},0,\frac{1}{\sqrt{10}},0,0,0,0,0,0,\frac{1}{\sqrt{10}},0\right)}_{\varphi^2}.
 \end{align*}}}
 It is easy to check that $\varphi^1, \varphi^2$ are eigenvectors of $L$ with eigenvalue $8$.
  This $\varphi$ does not have unit norm but we can compute  \eqref{eq:phipsi vector} to be $$\varphi_{\textup{Aut}(G)} = \left( \frac{1}{245} \left(30-8 \sqrt{11}\right), \frac{1}{245} \left(30-8 \sqrt{11}\right) \right).$$
 Therefore, $\varphi_{\textup{Aut}(G)}$ is  a certificate for the lower conformal rigidity of $G$. \qed
\end{example}

Even without symmetrized embeddings, the vectors $\varphi_\Psi$ can give us information about edge-isometric embeddings of $G$, as we see next. 

\begin{definition} \label{def:convex set phiH}
Define the convex set 
\begin{align}
    C(\Psi,\lambda) := \textup{conv} \left( \varphi_\Psi \,:\, \varphi \in \mathcal{E}_\lambda, \,\,\|\varphi\|=1 \right)
\end{align}
\end{definition}

Since $\{\varphi \in \mathcal{E}_\lambda, \|\varphi\|=1 \}$ is compact, its image under the polynomial map $\varphi \mapsto \varphi_\Psi$ is also compact, and 
$C(\Psi,\lambda)$ is a compact convex set in $\RR^{\mathcal{EO}_\Psi}$.

\begin{theorem} \label{thm:embeddingcharacterization}
    Suppose $G$ is vertex-transitive with respect to a subgroup $\Psi$ of $\textup{Aut}(G)$, and $\lambda>0$. Then $G$ has an edge-isometric embedding on $\mathcal{E}_\lambda$ if and only if the line $t {\bf 1}$ intersects the convex set ${C}(\Psi, \lambda)$.
\end{theorem}

\begin{proof}
Suppose we have an edge-isometric embedding $\mathcal{P}'$ on $\mathcal{E}_\lambda$, where we can assume that the columns of $P'$ are the  eigenvectors $a_1 \varphi^1,\ldots,a_d \varphi^d$, where $\|\varphi^i\|=1$ and $\sum a_i^2=1$, and the length of an edge in the embedding is $c >0$. For all $\sigma \in \Psi$, 
\begin{align}
    c^2 = \sum_{k=1}^d a_k^2 (\varphi^k_i - \varphi^k_j)^2 = \sum_{k=1}^d a_k^2 (\varphi^k_{\sigma(i)} - \varphi^k_{\sigma(j)})^2
\end{align}
 since $(i,j) \in E$ implies that $(\sigma(i), \sigma(j)) \in E$.
We $\Psi$-symmetrize $\mathcal{P}'$ by including the vectors $a_i\sigma \varphi^i$ for every $\sigma \in \Psi$ and every column $a_i\varphi^i$, attaining a new embedding $\mathcal{P}$ on $\mathcal{E}_\lambda$.
Let once more $p_i$ and $p_j$ be the embeddings of vertices $v_i$ and $v_j$ in $\mathcal{P}$. 
If $(i,j) \in E$, then
\begin{equation}
        \|p_i - p_j \|^2  
         = \sum_{\sigma \in \Psi} \sum_{k=1}^d a_k^2(\varphi^k_{\sigma(i)} -\varphi^k_{\sigma(j)} )^2 = \|\Psi\| c^2
\end{equation} 
and so, $\mathcal{P}$ is again, an edge-isometric embedding of $G$ on $\mathcal{E}_\lambda$. Moreover,
 \begin{equation} \label{eq:symmetrizededgelength}
        \|p_i - p_j \|^2  
         = \sum_{k=1}^d a_k^2\left(\sum_{\sigma \in \Psi} (\varphi^k_{\sigma(i)})^2 + \sum_{\sigma \in \Psi}  (\varphi^k_{\sigma(j)})^2 - 2 \sum_{\sigma \in \Psi}  \varphi^k_{\sigma(i)} \varphi^k_{\sigma(j)}\right).
\end{equation} 
This quantity being independent of the edge $(i,j)$ is equivalent to 
\begin{equation}\label{eq:orbitconvexcombination}
\left(\sum_{k=1}^d a_k^2\sum_{\sigma \in \Psi}  \varphi^k_{\sigma(i)} \varphi^k_{\sigma(j)}\right)_{\mathcal{EO}_\Psi} = \sum_{k=1}^d a_k^2 \varphi^k_\Psi.
\end{equation} 
being a constant vector. Since $\sum a_k^2=1$ this is a convex combination of $\{\varphi_\Psi^k\}$ and so it lies in ${C}(\Psi, \lambda)$.
Conversely, if there is a constant vector in ${C}(\Psi, \lambda)$, we can write it as $\sum_{k=1}^d b_k^2 \psi^k_\Psi$ for some $\psi^k \in \mathcal{E}_{\lambda}, \,\|\psi^k \|=1$ and $\sum b_k^2 =1$. Reversing \eqref{eq:orbitconvexcombination} and \eqref{eq:symmetrizededgelength} we get an edge-isometric embedding on $\mathcal{E}_\lambda$.
\end{proof}

\begin{corollary} \label{cor:cr_eq_formulation}
    Suppose $G$ is vertex-transitive with respect to a subgroup $\Psi$ of $\textup{Aut}(G)$ and $\lambda > 0$. Then $G$ is lower conformally rigid if and only if the line $t {\bf 1}$ intersects the convex set ${C}(\Psi, \lambda_2)$. The analogous statement holds for upper conformal rigidity. 
\end{corollary}


\begin{question}
Is it true that for a vertex-transitive graph, the existence of an edge-isometric embedding in the second (and last) eigenspace implies the existence of a symmetrized edge-isometric embedding in that eigenspace? 
\end{question}

If the answer to this question is ``yes", then there will exist a $\varphi$ in the relevant eigenspace such that $\varphi_\Psi$ certifies conformal rigidity as in Corollary~\ref{cor:generaleigenvectorcriterion}. So far we do not have an example where no eigenvector $\varphi$ certifies conformal rigidity but the condition in Corollary~\ref{cor:cr_eq_formulation} is satisfied.

\section{Cayley graphs on abelian groups}
\subsection{Cayley Graphs}
Cayley graphs are  vertex-transitive graphs with respect to their underlying groups  and hence we can apply the above results to them. We recall that a Cayley graph $\textup{Cay}(\Gamma,S)$ is defined by a group $(\Gamma, \circ)$ and a subset of elements $S \subset \Gamma$  known as the {\em generating set}. The vertex set $V(G) = \Gamma$ is given by the elements of the group. Edges are given by adding a directed edge from $g$ to $g \circ s$. Since we are only interested in undirected graphs, we will insist on $s \in S \implies s^{-1} \in S$. Such an $S$ is said to be {\em symmetric}. In this case, for any directed edge $(g, g\circ s)$, we also have the directed edge $(g \circ s, (g \circ s)\circ s^{-1})$, and merging them we have an undirected edge connecting $g$ and $g \circ s$. 

The automorphism group of the Cayley graph $\textup{Cay}(\Gamma,S)$ contains $\Gamma$ but can be bigger, and $\textup{Cay}(\Gamma,S)$ is vertex-transitive with respect to the action of $\Gamma$. If an edge $(g, g \circ s)$ is labeled $s$, then all edges labeled $s$ are in the same $\Gamma$-orbit.  Therefore, for an eigenvector $\varphi$ we have in this case
\begin{align} \label{eq:real eigvec symm}
\varphi_\Gamma := \left(   \sum_{g \in \Gamma}  \varphi(g) \varphi(g \circ s) \right)_{s \in S}.
\end{align}

This is precisely the vector used in 
\cite[Theorem 2.3]{steinthomas} to describe a curious sufficient criterion for a Cayley graph to be conformally rigid: it is enough to find a single eigenvector satisfying a type of algebraic identity. The proof in \cite{steinthomas} was motivated by the calculus of variations and some unexpected algebraic simplifications. The full statement of \cite[Theorem 2.3]{steinthomas} is as follows.

\begin{theorem}[Theorem 2.3, \cite{steinthomas}] \label{thm:Cayley sufficiency}
    Let $\textup{Cay}(\Gamma,S)$ be a finite Cayley graph, where $S$ is symmetric. If there exists an eigenvector  $\varphi$ corresponding to $\lambda_2$ such that 
    \begin{align} \label{eq:Cayley sum property}
     \sum_{g \in \Gamma}  \varphi(g) \varphi(g \circ s) \qquad \mbox{is independent of}~s,
    \end{align}
    then $\textup{Cay}(\Gamma,S)$ is lower conformally rigid (and the analogous statement for upper conformal rigidity also holds).
\end{theorem}

   We illustrate Theorem~\ref{thm:Cayley sufficiency} on the circulant graph on $n=18$ vertices with generating set $S = \left\{1,5\right\}$ (or, depending on notational custom,  $S = \left\{-5,-1,1,5\right\}$). The eigenspace corresponding to $\lambda_2$ is two-dimensional and spanned by $\left\{\varphi^1, \varphi^2\right\}$ where
\begin{align*}
   \varphi^1 &= (-1, -1, 0, 1, 1, 0, -1, -1, 0, 1, 1, 0, -1, -1, 0, 1, 1, 0) \\
   \varphi^2 &= (1, 0, -1, -1, 0, 1, 1, 0, -1, -1, 0, 1, 1, 0, -1, -1, 0, 1). 
\end{align*}
Theorem~\ref{thm:Cayley sufficiency} seeks an element $\varphi$ in the span of $\varphi^1,\varphi^2$ such that 
$$ \sum_{g \in G} \varphi(g) \varphi(g +1) =  \sum_{g \in G} \varphi(g) \varphi(g +5)$$
and a quick check shows that $\varphi = \varphi^1$ has that property. In all explicit examples that were tested, this sufficient condition also appeared to be necessary: conformal rigidity does seem to imply the existence of such a vector $\varphi$ and this was stated as an open problem in \cite{steinthomas}. In  Theorem~\ref{cor:generaleigenvectorcriterion} we provide a necessary and sufficient condition when allowing for complex-valued eigenvectors.

\subsection{Cayley Symmetrization.}
We now look at what the set $C(\Gamma,\lambda)$ looks like for a Cayley graph on an abelian group $\Gamma$. A crucial ingredient in the argument below is the special structure of eigenvectors of Cayley graphs (see, for example, Godsil \cite[\S 12, Lemma 9.2]{godsil}). In what follows denote by $\CC \mathcal{E}_{\lambda}$ the complex eigenspace of the Laplacian $L$ associated to $\lambda$, and for a complex eigenvector $\chi$ we denote
\begin{align} \label{gammaequation}
\chi_\Gamma:= \left(   \sum_{g \in \Gamma}  \chi(g) \overline{\chi(g \circ s) }\right)_{s \in S}
\end{align}
which strictly extends the definition in \eqref{eq:real eigvec symm} for real eigenvectors.

When the group $\Gamma$ is abelian, there is an explicit set of eigenvectors of the Laplacian of $\textup{Cay}(\Gamma,S)$ coming from the representation theory of $\Gamma$. A group homomorphism $\rho \,:\, \Gamma \rightarrow \textup{GL}(\CC)$ is a called a one-dimensional {\em representation} of $\Gamma$ and the vector $\chi_\rho = (\rho(g))_\Gamma$ is called the {\em character} of $\rho$. If $\Gamma$ is an abelian group, all representations of $\Gamma$ are one-dimensional and the characters of $\Gamma$ form an orthonormal basis of $\CC^{|\Gamma|}$, partitioned into bases for the complex eigenspaces of the Laplacian.

\begin{lemma} \label{lem:orthogonality} 
Suppose $\{\chi^1,\ldots,\chi^k\}$ is an orthonormal basis for $\CC\mathcal{E}_{\lambda}$ made of characters of $\Gamma$. Then, for $j \neq \ell$ and $s \in S$, we have
$$ \sum_{g \in \Gamma} \chi^j(g) \overline{\chi^{\ell}(g \circ s)} = 0.$$
\end{lemma}
\begin{proof}
The distinct characters of an abelian group $\Gamma$ are pairwise orthogonal under the inner product 
$$ \langle \chi^j, \chi^\ell \rangle := \sum_{g \in \Gamma} \chi^j(g) \overline{\chi^\ell(g)}. $$
Using that characters  $\chi$ are group homomorphisms, and thus $\chi(gh) = \chi(g) \chi(h)$, we see that
any two characters $\chi^j \neq \chi^\ell$ satisfy
$$ \sum_{g \in \Gamma} \chi^j(g) \overline{\chi^\ell(g \circ s)} = \overline{\chi^\ell(s)}  \sum_{g \in \Gamma} \chi^j(g) \overline{\chi^\ell(g)} = 0.$$
\end{proof}

With this, we get the  structural result  that $C(\Gamma,\lambda)$ is a polytope. 

\begin{theorem} \label{thm:cayleymain}
Consider a Cayley graph $\textup{Cay}(\Gamma,S)$ with $\Gamma$ abelian. Let $\{\chi^1,\ldots,\chi^k\}$ be an  orthonormal character basis of $\CC\mathcal{E}_{\lambda}$. Then, with $\chi^j_\Gamma \in \mathbb{C}^{|S|}$ as in (\ref{gammaequation}),
\begin{align*}
C(\Gamma,\lambda) &= \RR^{|S|} \cap \textup{conv}\left\{ \chi^j_\Gamma \, : \, j=1,\ldots,k \right\}\\
&=\RR^{|S|} \cap\{\varphi_\Gamma \, : \, \varphi \in {\CC}\mathcal{E}_{\lambda}, \|\varphi\|=1\}.
\end{align*}
\end{theorem}
\begin{proof}
Take any (real) $\varphi \in \mathcal{E}_{\lambda}$, $\|\varphi\|=1$. We can write $ \varphi = a_1 \chi^1 + \dots + a_k \chi^k$ for some complex numbers $a_i$. Since $\|\varphi\|=1$, $\sum_{j=1}^k |a_j|^2 = 1$.
By Lemma~\ref{lem:orthogonality}, we have
\begin{align*}
 \sum_{g \in \Gamma}\varphi(g)  \overline{\varphi(g \circ s)} &= \sum_{g \in \Gamma} \left( \sum_{j=1}^{k} a_j \chi^j(g) \right)   \left( \sum_{\ell=1}^{k} \overline{a_{\ell}} \overline{ \chi^{\ell}(g \circ s)} \right) \\
&= \sum_{j,\ell=1}^{k} a_j \overline{a_{\ell}} \sum_{g \in \Gamma} \chi^j(g) \overline{\chi^{\ell}(g \circ s)} = \sum_{j=1}^{k} |a_j|^2 \sum_{g \in \Gamma} \chi^j(g) \overline{\chi^{j}(g \circ s)},
\end{align*}
where the sum $\sum_{g \in \Gamma} \chi^j(g) \overline{\chi^{j}(g \circ s)}$ only depends on $j$ and $s$. We arrive at
$$\varphi_\Gamma := \left(   \sum_{g \in \Gamma}  \varphi(g) \varphi(g \circ s) \right)_{s \in S} = \sum_{j=1}^{k} |a_j|^2 \chi_\Gamma^j \in \mathbb{C}^{|S|}.$$
This means that $\varphi_\Gamma$ is a convex combination of the $\dim \mathcal{E}_{\lambda}$-many complex vectors $\chi^j_\Gamma$ in $\mathbb{C}^{|S|}$. Therefore, 
$C(\Gamma,\lambda)$ is contained in $\textup{conv} \{\chi_\Gamma^j, j=1, \ldots, k\} \cap \RR^S,$
and the first set is contained in the second:
$$ C(\Gamma,\lambda) \subseteq \RR^{|S|} \cap \textup{conv}\left\{ \chi^j_\Gamma \, : \, j=1,\ldots,k \right\}.$$

We now show that the second set is included in the third. If $x$ is a real vector that is a convex combination of the $\chi^j_\Gamma$ then we can write it as $$x=\sum_{j=1}^k a_j^2 \chi^j_\Gamma$$ for some real numbers $a_j$, and following the equalities in the reverse order we conclude $x=\varphi_\Gamma$ where $\varphi= \sum_{j=1}^k a_j \chi^j$. 

It remains to prove that the third set is contained in the first set. Let $\varphi \in \mathbb{C} \mathcal{E}_{\lambda}$ be such that $\|\varphi\|=1$ and $\varphi_\Gamma$ is real. We decompose $\varphi$ into  its real and imaginary parts using two scaled real unit vectors. More precisely, there exists real vectors $\varphi^1, \varphi^2$, both normalized in $\ell^2$, such that
$$ \varphi = \mbox{Re}(\varphi) + i \cdot \mbox{Im}(\varphi) = a_1 \varphi^1 +  i a_2 \varphi^2.$$
 Note that $a_1^2+a_2^2=\|\varphi\|^2=1$ and $\varphi^i \in \mathcal{E}_{\lambda}$. Then, since $\varphi_\Gamma$ is real
\begin{align*}
\varphi_\Gamma &= \mbox{Re} \left(\varphi_\Gamma \right)= \textup{Re}\left(\sum_{g \in \Gamma} \varphi(g) \overline{\varphi(g \circ s)}\right) \\
& =  \sum_{g\in \Gamma} a_1^2\varphi^1(g \circ s) {\varphi^1(g)} + a_2^2\varphi^2(g\circ  s) {\varphi^2(g)} \\
&=   a_1^2 \left( \sum_{g\in \Gamma} \varphi^1(g \circ s) {\varphi^1(g)} \right) + a_2^2 \left(\sum_{g\in \Gamma} \varphi^2(g \circ s) {\varphi^2(g)} \right).
 \end{align*}
Since $a_1^2 + a_2^2 = 1$, this shows that $\varphi_\Gamma$ is a convex combination of $\varphi^1_\Gamma$ and $\varphi^2_\Gamma$ and so it belongs to $C(\Gamma,\lambda)$.
\end{proof}

The first consequence of Theorem~\ref{thm:cayleymain} is that for Cayley graphs on abelian groups, the sufficient condition for conformal rigidity in \cite{steinthomas} is also necessary, if we allow for complex eigenvectors.

\begin{theorem}[Necessary and Sufficient Cayley Criterion]
   \label{thm:necc and suff condn cayley}
    Let $\Gamma$ be a finite abelian group and let $\textup{Cay}(\Gamma,S)$ be a Cayley graph (with $-S = S$). Then $\textup{Cay}(\Gamma,S)$ is lower conformally rigid if and only if there exists a complex eigenvector $\varphi$ corresponding to $\lambda_2$ such that 
    \begin{align} \label{eq:Cayley sum property2}
     \sum_{g \in \Gamma}  \varphi(g) \overline{\varphi(g \circ s)} \quad \mbox{is real, and independent of}~s.
    \end{align}
    The analogous statement for upper conformal rigidity also holds.
\end{theorem} 

It is a natural question whether one may  assume in Theorem~\ref{thm:necc and suff condn cayley} that $\varphi$ is a real eigenvector (as opposed to a complex-valued eigenvector/character).  If so, then Theorem~\ref{thm:necc and suff condn cayley} would be a full converse to \cite[Theorem 2.3]{steinthomas} in the abelian case. These conditions do not yet \textit{explain} when a Cayley graph is conformally rigid. The paper \cite{steinthomas} considers rather simple circulant graphs on two generators; whether or not these Cayley graphs are conformally rigid remains an open problem.

A second consequence of Theorem~\ref{thm:cayleymain} is that for Cayley graphs on abelian groups, if we know the characters of the underlying group, we can check conformal rigidity by simply solving a linear program. This gives us a very efficient and numerically robust way of certifying conformal rigidity in these cases. 
The property that $C(\Gamma,S)$ becomes a polytope when $\Gamma$ is abelian is similar to results in the literature where an optimization problem that is invariant under the action of a (abelian) group will {\em symmetry reduce} to a linear program. For example, see \cite{vallentincayley} for a very closely related setup where they use the full character basis of $\CC^{|\Gamma|}$ to derive a linear program for the stability number problem on $\textup{Cay}(\Gamma,S)$, while in our case, we are picking a character basis of a given eigenspace to produce a polytope/linear program. 
The famous Delsarte linear programming bound on the size of Boolean codes is another instance of a similar symmetry-reduction to a linear program. 
 
 \subsection{A Hyperplane Perspective.} We quickly discuss a slightly different perspective motivated by the {\em calculus of variations}.  We may think of the Graph Laplacian
 $$ L(w) = D(w) - A(w)$$
as a matrix indexed by the weights. Correspondingly, for $\varphi \in \mathcal{E}_{\lambda_2}$ normalized to $\| \varphi\| = 1$, we have
$$ \sum_{(i,j) \in E}  w_{ij} (\varphi(i) - \varphi(j))^2 =  \sum_{s \in S} \sum_{g \in \Gamma} w_{g,s} (\varphi(g \circ s) - \varphi(g))^2 = \lambda_2.$$
A basic symmetrization argument (already carried out in \cite{steinthomas}) shows that
the weight $w_{g,s}$ can be assumed to only depend on  $w_{g,s} =: w_s$. The question is now: can the weights be changed in such a way that this particular expression increases for \textit{every} $\varphi \in \mathcal{E}_{\lambda_2}$ simultaneously. Expanding the square and changing variables, we see that
\begin{align*}
    \sum_{s \in S} \sum_{g \in \Gamma} w_{s} (\varphi(g \circ s) - \varphi(g))^2 &=   \sum_{s \in S} \sum_{g \in \Gamma} w_{s} \left( \varphi(g \circ s)^2 + \varphi(g)^2 - 2\varphi(g)\varphi(g \circ s) \right) \\
    &= 2 \sum_{s \in S} w_s - 2\sum_{s \in S} w_s \sum_{g \in \Gamma} \varphi(g) \varphi(g \circ s)
\end{align*}
which leads us naturally to $\varphi_{\Gamma}$. The first sum is invariant (since we keep the total weight the same). The second sum can be seen as an inner product:
$$\sum_{s \in S} w_s \sum_{g \in \Gamma} \varphi(g) \varphi(g \circ s) = \left\langle w, \varphi_\Gamma \right\rangle.$$
The question is now whether, by changing the weights $w$, the quantity $\left\langle w, \phi_\Gamma \right\rangle$ can be decreased (since that would increase $\lambda_2$). Note that things are a little bit more complicated because the eigenvectors $\varphi \in \mathcal{E}_{\lambda_2}$ also depend on the weight $w$. Moreover, the quantity should decrease not only for a single eigenvector $\varphi$ but it should do so \textit{uniformly} for the entire eigenspace $\mathcal{E}_{\lambda_2}$ while simultaneously preserving orthogonality to $\mathbf{1}$ (as the total weight is unchanged). This naturally leads to the problem of understanding how the unit ball in the eigenspace $\mathcal{E}_{\lambda_2}$ behaves when mapped using the expression $\phi_G = \sum_{g \in \Gamma} \varphi(g) \varphi(g \circ s)$ and whether this intersects the one-dimensional eigenspace spanned by $\mathbf{1}$. If it does not, then not only is the graph \textit{not} conformally rigid but the hyperplane separation theorem will propose a direction of improvement: a way to change the weights so that $\lambda_2$ increases.

\subsection{A Family of Circulants}
In this section we prove the conformal rigidity of an infinite class of graphs using Theorem \ref{thm:cayleymain} and Theorem
\ref{thm:embeddingcharacterization}. We do not have an alternate proof strategy for this result.

We will focus on the case of abelian Cayley graphs, more concretely, on circulant graphs $G=\textup{Cay}(\mathbb{Z}_N,S)$. The characters of $\mathbb{Z}_N$ are of the form $\chi^k(i)=\omega^{ki}$ where $\omega$ is a primitive $N$-th root of unity, and these are therefore the complex eigenvectors of $G$. The corresponding eigenvalues are  
\[
\lambda_k = 2\sum_{j \in S} \left( 1 - \textup{Re}(\omega^{kj}) \right)= 2\sum_{j \in S} \left( 1 - \cos\left(\frac{2\pi k j}{n}\right) \right), \quad k = 0,1,\dots,N-1.
\]
We start by noting a simple corollary of our previous results.

\begin{corollary}
Let $\Gamma$ be an abelian group and $G=\textup{Cay}(\Gamma,S)$. For some character $\chi$ of $\Gamma$, if we have 
that $(\textup{Re}(\chi(s)))_{s \in S}$ is constant, then $G$ has an edge-isometric embedding on $\mathcal{E}_{\lambda}$, where $\lambda$ is the eigenvalue associated to $\chi$.
\end{corollary}
\begin{proof}
Just note that $\frac{1}{2}\chi_\Gamma+\frac{1}{2}\overline{\chi}_\Gamma$ has coordinate $s$ equal to $\textup{Re}(\chi(s))$, and so by Theorems \ref{thm:embeddingcharacterization} and \ref{thm:cayleymain} we have the result.
\end{proof}

In our case, if we fix the character $\chi^k$, we can guarantee that $G$ has an edge-isometric embedding on $\mathcal{E}_{\lambda_k}$ if $\textup{Re}(\omega^{ks})$ is constant over $s \in S$.

\begin{corollary} \label{cor:circulantcriterion}
Let $m,n>1$, and let $G=\textup{Cay}(\mathbb{Z}_{mn},\{1,n-1\})$. Then $G$ has an edge-isometric embedding on $\mathcal{E}_{\lambda_{am}}$ for any positive integer $a$.
\end{corollary}
\begin{proof}
Just note that $\omega^{am}$ and $\omega^{am(n-1)}=\omega^{-am}$ have the same real part.    
\end{proof}

We can use this result to show conformal rigidity of a large class of graphs. The only problem is to find the smallest and largest  of the nonzero eigenvalues $\lambda_k$. In what follows we will find them for $\mathbb{Z}_{3n}$ but the same idea works more generally. 

\begin{proposition} \label{prop:infinite circulants}
The circulant graphs $\textup{Cay}(\mathbb{Z}_{3n},\{1,n-1\})$, with $n \geq 6$, are conformally rigid.
\end{proposition}
\begin{proof}
Examining the sequence of $\lambda_k$ for a large $n$ (Figure \ref{fig:eigenvalues}) we immediately understand what is happening. The eigenvalues align in three different curves, depending on their congruence class modulo $3$.
\begin{figure}
    \centering
    \includegraphics[width=7cm]{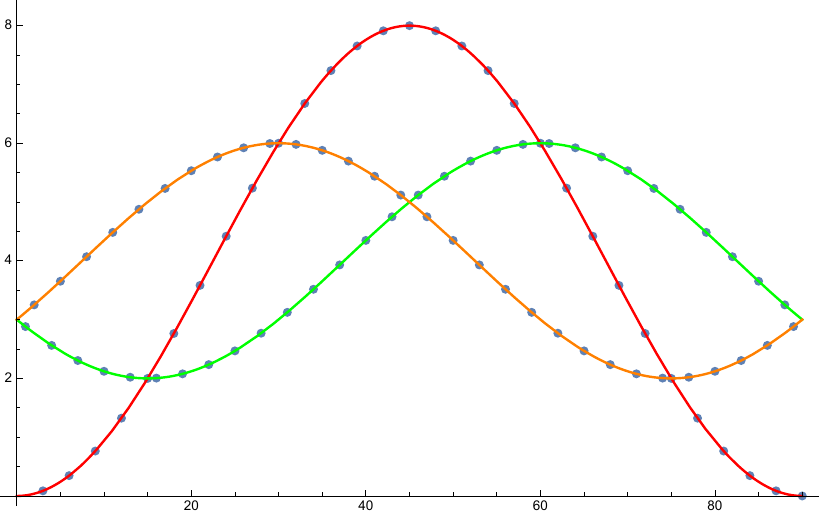}
    \caption{Eigenvalues $\lambda_k$ of $\textup{Cay}(\mathbb{Z}_{90},\{1,29\})$}
    \label{fig:eigenvalues}
\end{figure}
For $k$ a multiple of $3$ we get $\lambda_k=4 - 4\cos(2\pi k/(3n))$, for $k \cong 1 \mod 3$ we get $\lambda_k=4 - 2\sin(2\pi k/(3n) + \pi/6)$ and for  $k \cong -1 \mod 3$ we get $\lambda_k=4 + 2\sin(2\pi k/(3n) - \pi/6)$. This can be verified with standard trigonometric identities.
This implies that for $k$ not divisible by $3$, we have $2 \leq \lambda_k \leq 6$. For sufficiently large $n$ the minimum and maximum will then be attained by a multiple of $3$. The smallest value will be attained at $k=3$ and the largest at the middle value, $k=3\lfloor n/2 \rfloor.$ From Corollary \ref{cor:circulantcriterion} the result follows.
\end{proof}

\begin{example}
For odd $n$ we can see that the above construction always gives two embeddings from $G=\textup{Cay}(\mathbb{Z}_{3n},\{1,n-1\})$ to the vertices of a regular $n$-gon. It corresponds to a $3$ to $1$ map that rolls the $3n$-cycle in $G$ three times around an $n$-cycle. In Figure \ref{fig:embedding} we can see the color-coded embeddings for $n=7$.

\begin{figure}
    \centering
    \includegraphics[width=3.5cm]{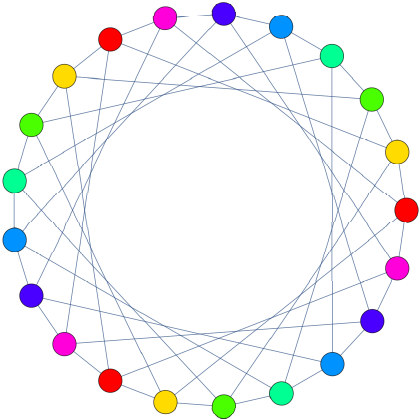} \hfill
    \includegraphics[width=2.7cm]{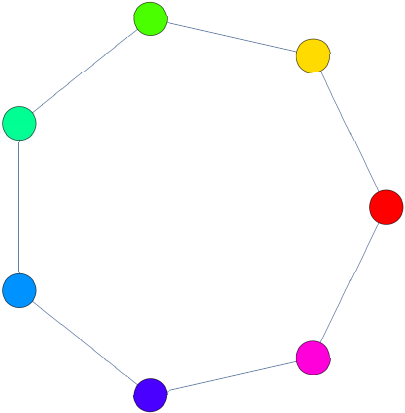} \hfill
    \includegraphics[width=2.7cm]{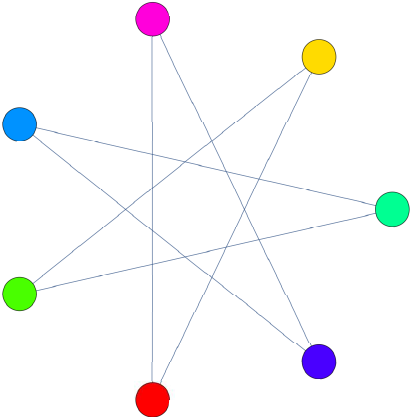}
    \caption{The graph $\textup{Cay}(\mathbb{Z}_{21},\{1,6\})$ and its two edge-isometric spectral embeddings to $\lambda_2$ and $\lambda_{\max}$.}
    \label{fig:embedding}
\end{figure}
\end{example}

It remains to be shown that these circulants  are not all trivially conformally rigid for other reasons, i.e., that they are not $1$-walk regular.

\begin{proposition} \label{prop:circulants not 1-walk regular}
 The circulant graph  $G_n := \textup{Cay}(\mathbb{Z}_{3n},\{1,n-1\})$ is $1$-walk regular if and only if $n\equiv -1 \mod 3$.
\end{proposition}
\begin{proof}

It is known (see for instance \cite[Theorem 5.1]{potovcnik2020recipes}) that $\textup{Cay}(\mathbb{Z}_N,\{1,k\})$ is edge-transitive if $k^2 \equiv \pm 1 \mod N$. If $n = 3k-1$, we have that $(n-1)^2 \equiv \pm 1 \mod n$ so $G$ is edge-transitive and hence, $1$-walk regular.
For the other cases, if we consider the number of paths of size $n-1$ from $0$ to $1$ and from $0$ to $n-1$ in $G_n$, a straightforward but somewhat cumbersome count shows that they differ by $\pm 1$ so the graphs are not $1$-walk regular.
\end{proof}

Propositions~\ref{prop:infinite circulants} and \ref{prop:circulants not 1-walk regular} answer a question raised in \cite{steinthomas}, as to whether there are infinitely many conformally rigid circulants that are not edge-transitive.

\section{A semidefinite programming interpretation}

In closing, we connect the certificate for conformal rigidity  from a symmetrized embedding to the  semidefinite programming (SDP) check for conformal rigidity (\S 3 and \S 6.3 in  \cite{steinthomas}). The connection to semidefinite programming can be observed from the set $C(\Psi,\lambda)$. We use $\mathcal{S}^k$ to denote the set of $k \times k$ symmetric matrices with entries in $\RR$, and $\mathcal{S}^k_+$ to denote its subset of positive semidefinite (psd) matrices.

\begin{lemma}
Suppose $G$ is vertex-transitive with respect to a subgroup $\Psi$ of $\textup{Aut}(G)$. Let $\varphi^1,\ldots,\varphi^k$ be a basis for $\mathcal{E}_{\lambda}$, $\lambda > 0$, forming the columns of a matrix $U$. Then
\begin{equation}\label{eq:sdpformulation}
C(\Psi,\lambda)=\left\{\left( \sum_{\sigma \in \Psi} Y_{\sigma(i) \sigma(j)} \right)_{\mathcal{EO}_\Psi} \, : \, Y=UXU^\top, X \in \mathcal{S}_+^k\right\}.
\end{equation}
\end{lemma}
\begin{proof}
Note that $Y=UXU^\top$ for some $X \in \mathcal{S}^k_+$ if and only if $Y=(UA)(UA)^T$ for some matrix $A = (a_{ij})$. This means that $Y = \sum_{\ell=1}^k \psi^\ell (\psi^\ell)^T$ where  
$$\psi^\ell := \sum_{i=1}^k a_{i\ell} \varphi^i.$$ 
Therefore, 
$$\left( \sum_{\sigma \in \Psi} Y_{\sigma(i) \sigma(j)} \right)_{\mathcal{EO}_\Psi}= \sum_{\ell=1}^k \psi_\Psi^\ell \in C(\Psi,\lambda).$$
On the other hand, for any $\varphi \in \mathcal{E}_{\lambda}$ we can write $\varphi = \sum a_i \varphi^i$, and by taking $X=aa^T$ we see that $\varphi_\Psi$ is in the right hand side of \eqref{eq:sdpformulation}. By convexity we get the opposite inclusion, proving equality.    
\end{proof}

Theorem \ref{thm:embeddingcharacterization} can now be translated into the language of SDP feasibility.

\begin{proposition}\label{prop:sdpformulation}
    Suppose $G$ is vertex-transitive with respect to a subgroup $\Psi$ of $\textup{Aut}(G)$, and $\lambda>0$. Let $(i_1,j_1),\ldots,(i_\ell,j_\ell)$ be representatives of edge orbits of  $G$ under $\Psi$. Then $G$ has an edge-isometric embedding on $\mathcal{E}_\lambda$ if and only there exists $X \in \mathcal{S}_+^k$ such that $\textup{tr}(UXU^\top)=1$ and 
    $$\sum_{\sigma \in \Psi} (UXU^\top)_{\sigma(i_1) \sigma(j_1)} = \sum_{\sigma \in \Psi} (UXU^\top)_{\sigma(i_k) \sigma(j_k)}, \textrm{ for } k=2,...,\ell.  $$
\end{proposition}
Note that the condition on the trace is to avoid the trivial case $X=0$, that does not correspond to an embedding. This is a SDP feasibility problem with number of constraints equal to the 
number of edge orbits of $G$ under $\Psi$. This gives us another case where the sufficient condition given by Corollary \ref{cor:generaleigenvectorcriterion} is actually necessary.

\begin{theorem} \label{thm:twoedge}
    Suppose $G$ is vertex-transitive with respect to a subgroup $\Psi$ of $\textup{Aut}(G)$, with at most two edge orbits with respect to that action. Then $G$ has an edge-isometric embedding on $\mathcal{E}_{\lambda}$ with $\lambda>0$ if and only if there exists $\varphi \in \mathcal{E}_{\lambda}\setminus \{0\}$ such that $\varphi_\Psi$ is a constant vector.
\end{theorem}

\begin{proof}
We just need to prove that if there is any edge-isometric embedding  we can produce such a $\varphi$, since the reciprocal implication is a consequence of Theorem \ref{thm:embeddingcharacterization}. 

Since there is an edge-isometric embedding, the SDP feasibility problem of Proposition \ref{prop:sdpformulation} has a solution. The simplest case of the Barvinok-Pataki bound \cite[Theorem 1.1]{barvinok} on the ranks of feasible solutions to SDPs, tells us that any feasible SDP with two equality constraints has a solution with rank one. Let $X=aa^t$ be that solution.
Then $\varphi=Ua=\sum a_i \varphi^i \in \mathcal{E}_{\lambda}$. Moreover the SDP constraints of Proposition \ref{prop:sdpformulation} translate to $\|\varphi\|=1$ and $\varphi_\Psi$ being constant,  proving the result.
\end{proof}
An example  where Theorem~\ref{thm:twoedge} applies is the circulant $\textup{Cay}(\mathbb{Z}_{18},\{1,5\})$.

\end{document}